\documentclass[reqno,11pt]{article}

\usepackage{a4wide}
\usepackage{geometry}

\usepackage{hyperref}
\usepackage{adjustbox}
\usepackage{graphicx}
\usepackage{subfigure}
\usepackage[toc,page]{appendix}
\usepackage{pst-all}
\usepackage[english]{babel}
\usepackage{pdfsync}
\usepackage{amsmath,amssymb,amsthm}
\usepackage{mdwlist}
\usepackage{mathrsfs}
\usepackage{mathtools}
\usepackage{bm}
\usepackage{dsfont}

\newcommand{\dint}{\mathrm{d}}
\newcommand{\dist}{\mathsf{d}}
\newcommand{\Lint}{\mathit{L}}
\newcommand{\Sp}{\mathrm{X}}
\newcommand{\mea}{\mathfrak{m}}

\newcommand{\mms}{(\Sp, \dist, \mea)}
\newcommand{\mmso}{(\Sp_1, \dist_1, \mea_1)}
\newcommand{\mmst}{(\Sp_2, \dist_2, \mea_2)}
\newcommand{\proba}{\mathscr{P}}
\newcommand{\real}{\mathbb{R}}
\newcommand{\N}{\mathbb{N}}
\newcommand{\W}{\mathit{W}^{1,2}}
\newcommand{\WT}{\mathit{W}^{2,2}}

\newcommand{\Tan}{\mathit{T}\mathrm{X}}

\newcommand{\ener}{\mathsf{E}}
\newcommand{\hess}{\mathrm{Hess}}
\newcommand{\Div}{{\rm div}}

\newcommand{\Test}{\rm{Test}}

\newcommand{\F}{{\sf T}}

\newcommand{\D}{\mathrm{D}}

\newcommand{\Leb}{\mathscr{L}}
\newcommand{\Bor}{\mathscr{B}}
\newcommand{\Sclass}{\mathcal{S}^2}

\newcommand{\loc}{\mathsf{loc}}

\newcommand{\bpi}{\boldsymbol{\pi}}

\newcommand{\mae}{\mea\text{-a.e.}}

\newcommand{\Fl}{{\rm Fl}}

\DeclarePairedDelimiter{\abs}{\lvert}{\rvert}
\DeclarePairedDelimiter{\norma}{\lVert}{\rVert}

\def\Xint#1{\mathchoice 
  {\XXint\displaystyle\textstyle{#1}}%
  {\XXint\textstyle\scriptstyle{#1}}%
  {\XXint\scriptstyle\scriptscriptstyle{#1}}%
  {\XXint\scriptscriptstyle\scriptscriptstyle{#1}}%
  \!\int} 
\def\XXint#1#2#3{{\setbox0=\hbox{$#1{#2#3}{\int}$} 
  \vcenter{\hbox{$#2#3$}}\kern-.5\wd0}} 
\def\-int{\Xint -}

\numberwithin{equation}{subsection}

\newcommand{\R}{\mathbb{R}}
\newcommand{\Z}{\mathbb{Z}}

\newcommand{\mm}{{\mbox{\boldmath$m$}}}

\newcommand{\ppi}{{\mbox{\boldmath$\pi$}}}

\newcommand{\sfd}{{\sf d}}

\newcommand{\sfh}{{\sf h}}

\newcommand{\Id}{{\rm Id}}                          
\newcommand{\Kliminf}{K\kern-3pt-\kern-2pt\mathop{\rm lim\,inf}\limits}  
\newcommand{\supp}{\mathop{\rm supp}\nolimits}   
\newcommand{\Lip}{\mathop{\rm Lip}\nolimits}          
\renewcommand{\d}{{\mathrm d}}

\newcommand{\restr}[1]{\lower3pt\hbox{$|_{#1}$}} 

\newcommand{\la}{\left<}                  
\newcommand{\ra}{\right>}
\newcommand{\eps}{\varepsilon}  
\newcommand{\nchi}{{\raise.3ex\hbox{$\chi$}}}

\newcommand{\limi}{\varliminf}
\newcommand{\lims}{\varlimsup}

\newcommand{\fr}{\penalty-20\null\hfill$\blacksquare$}                      

\newcommand{\prob}[1]{\mathscr P(#1)}                   
\newcommand{\e}{{\rm{e}}}                           

\renewcommand{\mm}{\mathfrak m}                                

\renewenvironment{proof}{\removelastskip\par\medskip   
\noindent{\em proof} \rm}{\penalty-20\null\hfill$\square$\par\medbreak}

\newtheorem{theorem}{Theorem}[section]

\newtheorem{corollary}[theorem]{Corollary}
\newtheorem{lemma}[theorem]{Lemma}
\newtheorem{proposition}[theorem]{Proposition}
\newtheorem{assumption}[theorem]{Assumption}

\newtheorem{thmdef}[theorem]{Theorem/Definition}

\newtheorem{definition}[theorem]{Definition}

\newtheorem{remark}[theorem]{Remark}

\newcommand{\test}[1]{{\rm Test}(#1)}

\newcommand{\X}{{\rm X}}

\newcommand{\h}{{\sfh}}

\renewcommand{\ae}{{\textrm{\rm{-a.e.}}}}

\newcommand{\RCD}{{\sf RCD}}

\newcommand{\sqb}[1]{[#1]}

\newcommand{\lip}{{\rm lip}}
\newcommand{\M}{{\mathscr M}}

\newcommand{\HS}{{\lower.3ex\hbox{\scriptsize{\sf HS}}}}
\renewcommand{\H}[1]{{\rm Hess}(#1)}
\renewcommand{\div}{{\rm div}}

\newcommand{\dmin}{{\sf dim}_{\rm min}(\X)}
\newcommand{\dmax}{{\sf dim}_{\rm max}(\X)}

\title{Recognizing the flat torus among $\RCD^*(0,N)$ spaces via the study of the first cohomology group}
\author{Nicola Gigli \thanks{SISSA, Trieste. email: ngigli@sissa.it} \quad Chiara Rigoni  \thanks{SISSA, Trieste. email: crigoni@sissa.it}}

\makeindex

\begin{document}  

\maketitle

\begin{abstract}
We prove that if the dimension of the first cohomology group of a $\RCD^*(0,N)$ space is $N$, then the space is a flat torus.

This generalizes a classical result due to Bochner to the non-smooth setting and also provides a first example where the  study of the cohomology groups in such synthetic framework leads to geometric consequences.
\end{abstract}

\tableofcontents

\section{Introduction}
A classical result due to Bochner concerning manifolds with non-negative Ricci curvature is:
\begin{theorem}[Bochner]\label{thm:boch}
Let $M$ be a compact, smooth and connected  Riemannian manifold with non-negative Ricci curvature. Then:
\begin{itemize}
\item[i)] The dimension of the first cohomology group is bounded above by the dimension of the manifold,
\item[ii)] If these two dimensions are equal, then $M$ is a flat torus.
\end{itemize}
\end{theorem}
The key observation that leads to $(i)$ is the fact that under the stated assumptions every harmonic 1-form must be parallel and thus determined by its value at any given point $x\in M$. Since  Hodge's theorem grants that the $k$-th cohomology group is isomorphic to the space of harmonic k-forms, the claim follows.

For $(ii)$, the typical argument starts with the observation that if an $n$-dimensional manifold admits $n$ independent parallel vector fields, then such manifold must be flat. Hence its universal cover, equipped with the pullback of the metric tensor, must be the Euclidean space and the fundamental group $\pi_1(M)$ acts on it via isometries. Since $M\sim \R^n/\pi_1(M)$, all is left to show is that $\pi_1(M)\sim\Z^n$, which can be obtained by `soft' considerations about the structure of the isometries of $\R^n$ and the fact that $\R^n/\pi_1(M)$ is, by assumption, compact and smooth (see e.g. \cite{Petersen16} for the details).

\bigskip

This paper is about the generalization of the above result to the non-smooth setting of $\RCD^*(0,N)$ spaces (\cite{AmbrosioGigliSavare11-2}, \cite{Gigli12}). Our starting point is the paper \cite{Gigli14} by the first author where a differential calculus on such spaces has been built. Among other things, the vocabulary proposed there allows to speak of vector fields, $k$-forms, covariant derivative,  Hodge laplacian and cohomology groups $H^k_{\rm dR}$. In particular, a quite natural version of Hodge's theorem exists in this non-smooth setting, so that we know that cohomology classes are in correspondence with their unique harmonic representative. The basic structure around which the theory is built - and which offers a counterpart for the space of $L^2$ section of a normed vector bundle over a smooth manifold - is the one of $L^2$-normed $L^\infty$-module.

In searching for an analogous of Theorem \ref{thm:boch} in the nonsmooth setting, one thing to discuss is the notion of dimension. To this aim, let us recall that given a generic $L^2$-normed $L^\infty$-module $\mathscr M$ over a space $(\X,\sfd,\mm)$, there exists a unique, up to negligible sets, Borel partition $(E_i)_{i\in\N\cup\{\infty\}}$  of $\X$ such that for every $i\in\N$ the restriction of $M$ to $E_i$ has dimension $i$, and for no $F\subset E_\infty$ with positive measure the restriction of $\mathscr M$ to $F$ has finite dimension. We consider such partition $(E_i)$ for $M$ being the tangent module of the space $\X$ and think to the dimension of tangent module on $E_i$ as the dimension of the space $\X$ on the same set. Hence we  shall call $\dmin$ (resp. $\dmax$) the minimal (resp. supremum) of indexes $i\in\N\cup\{\infty\}$ such that $\mm(E_i)>0$. Then in \cite{Gigli14} the following result has been obtained:
\begin{theorem}\label{thm:hodgercd}
Let $(\X,\sfd,\mm)$ be a $\RCD^*(0,\infty)$ space. Then ${\rm dim}(H^k_{\rm dR})(\X)\leq \dmin$. 
\end{theorem}
The proof of this fact closely follows the argument for point $(i)$  in Theorem \ref{thm:boch}: every harmonic form is proved to be parallel, so that the dimension of the first cohomology group is bounded by how many independent (co)vector fields we can find on any region of our space. Notice that the compactness assumption is not present because, in the terminology of \cite{Gigli14}, harmonic forms are by definition in $L^2$.

We also point out that a priori such result might be empty, in the sense that without any additional assumption it is very possible that $\dmin=\infty$. In fact, the natural assumption on the space $\X$ is not that it is a $\RCD^*(0,\infty)$ space, but rather a $\RCD^*(0,N)$ one: given that the number $N\in[1,\infty]$ represents, in some sense, an upper bound for the dimension of the space we expect it to bound from above $\dmax$. This is indeed the case, as it has been proved by Han in \cite{Han14} (see also \cite{GP16} for an alternative argument):
\begin{theorem}\label{thm:han}
Let $(\X,\sfd,\mm)$ be a $\RCD^*(K,N)$ space. Then ${\dmax}\leq N$.
\end{theorem}

Coupling Theorems \ref{thm:hodgercd} and \ref{thm:han} we deduce that:
\begin{proposition}
Let $(\X,\sfd,\mm)$ be a $\RCD^*(0,N)$ space. Then  ${\rm dim}(H^k_{\rm dR})(\X)\leq N$.
\end{proposition}
This statement is a perfect analogue of point $(i)$ in Theorem \ref{thm:boch}. The aim of the present manuscript is to complete this analogy between the smooth and nonsmooth setting by proving:
\begin{theorem}\label{thm:mainintro}
Let $(\X,\sfd,\mm)$ be a $\RCD^*(0,N)$ space  such that ${\rm dim}(H^k_{\rm dR})(\X)= N$ (so that in particular $N$ is integer). Then it is isomorphic to the flat $N$-dimensional torus.
\end{theorem}
Here `isomorphic' means that there exists a measure preserving isometry from the torus  equipped with its Riemannian distance and a constant multiple of the induced volume measure to our space. Unlike the proof of Theorem \ref{thm:hodgercd} which closely mimics the original one, here we cannot  adapt the `smooth arguments' to the current setting: the problem is that it is not known whether $\RCD$ spaces admit a universal cover or not (there are some results in this direction - see \cite{MondinoWei16} - but it is unclear to us whether they can effectively be used for our current purposes).

We therefore have to  pursue a different strategy, the starting steps of our argument being:
\begin{itemize}
\item[-] We start studying the flow of an harmonic vector field on our space and, using the fact that in particular such vector field must be   parallel and divergence-free, we prove that  in accordance with the smooth case such flow is made of measure preserving isometries. Here by `flow' we intend in fact `Regular Lagrangian Flow' in the sense of Ambrosio-Trevisan \cite{Ambrosio-Trevisan14} who adapted to the setting of $\RCD$ spaces the analogous notion developed by Ambrosio in \cite{Ambrosio04} in connection with the Di Perna-Lions theory \cite{DiPerna-Lions89}. We remark that our appears to be the first application of Ambrosio-Trevisan theory to vector fields which are not gradients.
\item[-] We prove that given two such vector fields $X,Y$, for their flows $\Fl^X_t$ and $\Fl^Y_s$ we have the formula $\Fl^X_t\circ \Fl^Y_s=\Fl^{tX+sY}_1$ for any $t,s\in\R$.
\item[-] Our assumption on the space $(\X,\sfd,\mm)$ grants that there are $N$ independent and orthogonal vector fields $X_1,\ldots,X_N$ which are parallel and divergence-free, hence we can define the  map $\F:\X\times\R^N\to \X$  by
\[
(x,a_1,\ldots,a_N)\quad\mapsto\quad \Fl^{X_1}_{a_1}(\cdots \Fl^{X_N}_{a_N}(x)).
\]
What previously proved ensures that this map can be seen as an action of $\R^N$ on $\X$ by isomorphisms.
\end{itemize}
Analyzing the properties of the map $\F$ will lead to the desired isomorphism with the torus. The hardest part will be the proof of the fact that the action is transitive: to obtain this  will require a sharpening of the calculus tools available in the nonsmooth setting and, in particular, we will analyze the structure of the (co)tangent modules on product spaces which we believe to be interesting on its own.

\bigskip

We conclude recalling that Honda proved in \cite{Honda14} that the dimension of the first cohomology group is upper semicontinuous along a non-collapsing sequence of manifolds with same dimension and a uniform lower bound on the Ricci. This result hints at the possibility of obtaining an almost rigidity statement of our theorem in the context of $\RCD$ spaces, which would informally read as
\begin{quote} 
`if a $\RCD$ space almost fulfils the assumption of Theorem \ref{thm:mainintro}, then it is mGH-close to a flat torus'.
\end{quote}
In this direction it is worth to emphasize that in the smooth category more is known: as Colding proved in \cite{Colding97},  $N$-dimensional manifolds with Ricci $\geq -\eps$ and first cohomology group of dimension $N$, not only must be mGH-close to the torus, but also  homeomorphich to it if $N\neq 3$ and homotopic if $N=3$. Such topological information is out of reach of simple arguments based on the mGH-compactness of the class of $\RCD$ spaces.

\section{Preliminaries}
To keep the presentation at reasonable length we shall assume the reader familiar with the notions of Sobolev functions (see \cite{Cheeger00}, \cite{Shanmugalingam00}, \cite{AmbrosioGigliSavare11}), of differential calculus on metric measure spaces (see \cite{Gigli14}, \cite{Gigli17}) and with the notion of Regular Lagrangian Flow on metric measure spaces (\cite{Ambrosio-Trevisan14}, \cite{AT15}). Here we shall only recall a few facts mainly to fix the notation.

For us a metric measure space $\mms$ will always be a complete and separable metric space equipped with a reference non-negative (and non-zero) Borel measure $\mm$ which is finite on bounded sets and with full support, i.e. $\mm(\Omega)>0$ for every non-empty open set $\Omega\subset \X$. This latter assumption is not really necessary, but simplifies some statements.

\subsection{$L^2$- and $L^0$-normed modules and basis of differential calculus}
\begin{definition}[$L^2(\X)$-normed modules]
A $L^2(\X)$-normed $L^\infty(\X)$-module, or simply a $L^2(\X)$-normed module, is a structure $(\M,\|\cdot\|,\cdot,|\cdot|)$ where
\begin{itemize}
\item[i)] $(\M,\|\cdot\|)$ is a Banach space
\item[ii)] $\cdot$ is a bilinear map from $L^\infty(\X)\times \M$ to $\M$, called multiplication by $L^\infty(\X)$ functions, such that
\begin{subequations}
\label{eq:moltf}
\begin{align}
\label{eq:m0}f\cdot(g\cdot v)&=(fg)\cdot v,\\
\label{eq:m1}{\mathbf 1}\cdot v&=v,
\end{align}
\end{subequations}
for every $v\in \M$ and $f,g\in L^\infty(\X)$, where ${\bf 1}$ is the function identically equal to 1.
\item[iii)] $|\cdot|$ is a map from $\M$ to $L^2(\X)$, called pointwise norm, such that 
\begin{subequations}
\label{eq:pontnorm}
\begin{align}
\label{eq:normgeq} |v|&\geq 0\qquad\mm\ae\\
\label{eq:normpunt} |fv|&=|f|\,|v|\qquad\mm\ae\\
\label{eq:recnorm} \|v\|&=\sqrt{\int|v|^2\,\d\mm},
\end{align}
\end{subequations}
\end{itemize}
An isomorphism between  two $L^2(\X)$-normed modules is a linear bijection which preserves the norm, the product with $L^\infty(\X)$ functions and the pointwise norm.
\end{definition}
\begin{definition}[$L^0$-normed module]\label{def:l0mod} A $L^0$-normed module is a structure $(\M,\tau,\cdot,|\cdot|)$ where:
\begin{itemize}
\item[i)] $\cdot$ is a bilinear map, called multiplication with $L^0$ functions, from $L^0(\X)\times \M$ to $\M$ for which \eqref{eq:m0}, \eqref{eq:m1} hold for any $f\in L^0(\X)$, $v\in \M$,
\item[ii)] $|\cdot|:\M\to L^0(\X)$, called pointwise norm, satisfies \eqref{eq:normgeq} and \eqref{eq:normpunt}  for any $f\in L^0(\X)$, $v\in \M$,
\item[iii)] for some Borel partition $(E_i)$ of $\X$ into sets of finite $\mm$-measure, $\M$ is complete w.r.t.\ the distance 
\begin{equation}
\label{eq:m0dist}
\sfd_0(v,w):=\sum_i\frac{1}{2^i\mm(E_i)}\int_{E_i}\min\{1,|v-w|\}\,\d\mm
\end{equation}
and $\tau$ is the topology induced by the distance.
\end{itemize}
An isomorphims of $L^0$-normed modules is a linear homeomorphism preserving the pointwise norm and the multiplication with $L^0$-functions.
\end{definition}
It is readily checked that the choice of the partition $(E_i)$ in $(iii)$ does not affect the completeness of $\M$ nor the topology $\tau$.

\begin{thmdef}[$L^0$ completion of a module]\label{thm:defl0} Let $\M$ be a $L^2$-normed module. Then there exists a unique couple $(\M^0,\iota)$, where $\M^0$ is a $L^0$-normed module and $\iota:\M\to \M^0$ is linear, preserving the pointwise norm and  with dense image. 

Uniqueness is intended up to unique isomorphism, i.e.:  if $(\tilde \M^0,\tilde\iota)$ has the same properties, then there exists a unique isomorphism $\Phi:\M^0\to \tilde \M^0$ such that $\tilde\iota=\Phi\circ\iota$.
\end{thmdef}

\subsection{Sobolev spaces for locally integrable objects}
Given a metric measure space $\mms$, by  $L^2_\loc(\X)$ we mean the space of (equivalence classes w.r.t.\ $\mm$-a.e.\ equality of) Borel functions $f:\Sp\to\R$ such that $\nchi_Bf\in L^2(\Sp)$ for every bounded Borel set $B\subset \Sp$. A curve $t\mapsto f_t\in L^2_\loc(\Sp)$ will be called continuous (resp.\ absolutely continuous, Lipschitz, $C^1$) provided for any bounded Borel set $B\subset \Sp$ the curve $t\mapsto\nchi_B f_t\in L^2(\Sp)$ is continuous (resp.\ absolutely continuous, Lipschitz, $C^1$).

\bigskip

Recall that $\ppi\in\prob{C([0,1],\Sp)}$ is called {\bf test plan} provided 
\[
\begin{split}
(\e_t)_*\ppi&\leq C\mm\qquad\forall t\in[0,1],\qquad\text{ for some $C>0$},\\
\iint_0^1|\dot\gamma_t|^2\,\d t\,\d\ppi&<\infty.
\end{split}
\]
The {\bf Sobolev class} $\Sclass(\Sp)$ (resp.\ $\Sclass_\loc(\Sp)$) is the space of all Borel functions $f:\Sp\to \R$ for which there exists $G\in L^2(\Sp)$ (resp.\ $G\in L^2_\loc(\Sp)$) non-negative, called weak upper gradient, such that
\[
\int|f(\gamma_1)-f(\gamma_0)|\,\d\ppi(\gamma)\leq \iint_0^1G(\gamma_t)|\dot\gamma_t|\,\d t\,\d\ppi(\gamma) \qquad\text{ for every test plan }\ppi.
\]
It turns out that $f\in \Sclass_\loc(\Sp)$ and $G$ is a weak upper gradient  if and only if for every test plan $\ppi$ we have that for $\ppi$-a.e.\ $\gamma$ the map $t\mapsto f(\gamma_t)$ is in $W^{1,1}(0,1)$ and 
\begin{equation}
\label{eq:sobcurve}
|\frac\d{\d t}f(\gamma_t)|\leq G(\gamma_t)|\dot\gamma_t|\qquad a.e.\ t\in[0,1].
\end{equation}
From this characterization it follows that there exists a minimal weak upper gradient in the $\mm$-a.e.\ sense: it will be called {\bf minimal weak upper gradient} and denoted by $|\D f|$. With a simple cut-off argument, \eqref{eq:sobcurve} also shows that $f\in\Sclass_\loc(\Sp)$ if and only if $\eta f\in\Sclass(\Sp)$ for every $\eta$ Lipschitz and with bounded support.

The {\bf Sobolev space} $\W(\Sp)$ (resp.\ $\W_\loc(\Sp)$) is defined as $L^2\cap\Sclass(\Sp)$ (resp.\ $L^2_\loc\cap\Sclass_\loc(\Sp)$) and again one can check that $f\in \W_\loc(\Sp)$ if and only if $\eta f\in \W(\Sp)$  for every $\eta$ Lipschitz and with bounded support. $\W(\Sp)$  is a Banach space when endowed with the norm
\[
\|f\|_{\W(\Sp)}^2:=\|f\|_{L^2(\Sp)}^2+\||\D f|\|_{L^2(\Sp)}^2.
\]
$\mms$ is said {\bf infinitesimally Hilbertian} provided $\W(\Sp)$ is a Hilbert space.

Among others, minimal weak upper gradients have the following important locality property:
\[
|\D f|=|\D g|\qquad \mm-a.e.\text{ on }\{f=g\}\qquad\qquad\forall f,g\in\Sclass_\loc(\Sp).
\]
Also, it will be useful to keep in mind that
\begin{equation}
\label{eq:apprsloc}
\begin{split}
\forall f\in\Sclass_\loc(\Sp)\text{ there exists $(f_n)\subset \Sclass(\Sp)$ such that $\mm(\Sp\setminus\cup_n\{f=f_n\})=0$}\\
\text{and for every $n\in\N$ the function $f_n$ is bounded and with bounded support.}
\end{split}
\end{equation}
Such sequence can be obtained noticing that for $f\in\Sclass_\loc(\Sp)$ the truncated function $(f\wedge (-c))\lor c$ also belongs to $\Sclass_\loc(\Sp)$ for every $c>0$ and then proceeding with a cut-off argument.

From the notion of minimal weak upper gradient it is possible to extract the one of differential via the following result:
\begin{thmdef}\label{thm:defd}
There exists a unique couple $(L^0(T^*\Sp),\d)$ with $L^0(T^*\Sp)$ being a $L^0(\Sp)$-normed module and $\d:\Sclass_\loc(\Sp)\to L^0(T^*\Sp)$ linear and such that
\begin{itemize}
\item[i)] $|\d f|=|\D f|$ $\mm$-a.e.\ for every $f\in \Sclass_\loc(\Sp)$,
\item[ii)] $L^0(T^*\Sp)$ is generated by $\{\d f : f\in\Sclass_\loc(\Sp)\}$, i.e.\ $L^0$-linear combinations of objects of the form $\d f$ are dense in $L^0(T^*\Sp)$.
\end{itemize}
Uniqueness is intended up to unique isomorphism, i.e.\ if $(\M,\d')$ is another such couple, then there is a unique isomorphism $\Phi \colon L^0(T^*\Sp)\to \M$ such that $\Phi(\d f)=\d'f$ for every $f\in\Sclass_\loc(\Sp)$.
\end{thmdef}
Some remarks:
\begin{itemize}
\item[a)] Using the approximation property \eqref{eq:apprsloc} one can show that $\Sclass_\loc(\Sp)$ can be replaced with either one of $\Sclass(\Sp)$, $\W(\Sp)$ in the above statement
\item[b)] If one chooses to replace   $\Sclass_\loc(\Sp)$ with either $\Sclass(\Sp)$ of $\W(\Sp)$, then it is also possible to replace the $L^0$-normed module with a $L^2$-normed module in the statement and in this case in $(ii)$ `$L^0$-linear' should be replaced by `$L^\infty$-linear' (notice that the choice of the module also affects the topology considered, whence the possibility of having two different uniqueness results).  The proof is unaltered: compare for instance Theorem \ref{pbmod} with the construction of pullback module given in \cite{Gigli14} and \cite{Gigli17}. 
\item[c)]  Call $(L^2(T^*\Sp),\underline \d)$ the outcome of Theorem \ref{thm:defd} written for $L^2$-normed modulus and one of the spaces $\Sclass(\Sp),\W(\Sp)$. Then its $L^0$-completion can be fully identified with the couple $(L^0(T^*\Sp),\d)$ given by Theorem \ref{thm:defd}  in the sense that: there is a unique linear map $\iota:L^2(T^*\Sp)\to L^0(T^*\Sp)$ sending $\underline \d f$ to $\d f$ and preserving the pointwise norm, moreover such map has dense image.

This is trivial to check from the definitions and for this reason we won't use a distinguished notation for the differential coming from the `$L^2$' formulation of the statement. 
\end{itemize}
Let us now discuss other differentiation operators defined for objects with $L^2_\loc$ integrability.

The space of vector fields $L^0(T\Sp)$ is defined as the dual of the $L^0$-normed module $L^0(T^*\Sp)$. Equivalently, it is the $L^0$-completion of the dual $L^2(T\Sp)$ of the $L^2$-normed module $L^2(T^*\Sp)$ (see \cite{Gigli14}, \cite{Gigli17}). $L^2_\loc(T\Sp)\subset L^0(T\Sp)$ is the space of $X$'s such that $|X|\in L^2_\loc(\Sp)$.

We say that  $X \in \Lint^2_\loc (\Tan)$ has {\bf divergence} in $\Lint^2_\loc$, and  write $X \in \D(\Div_{\loc})$, if there exists $h \in \Lint^2_\loc(\Sp)$ such that for every $f \in \W(\Sp)$ with bounded support it holds
\[  
\displaystyle \int f h \ \dint \mea = - \int \dint f(X) \ \dint \mea.  
\]
In this case we put $\div X:= h$.

\bigskip

Let us now assume that $\mms$ is infinitesimally Hilbertian, so that the pointwise norms in $L^0(T^*\Sp)$, $L^0(T\Sp)$ induce pointwise scalar products. Recall that in this case the modules $\Lint^0(T^*\Sp)$ and $\Lint^0(T\Sp)$ are canonically isomorphic via the Riesz (musical) isomorphism
\begin{equation}
\label{eq:music}
\flat \colon \Lint^0(\mathit{T} \Sp) \rightarrow \Lint^0(\mathit{T}^\ast \Sp) \qquad \ \text{and} \qquad \ \sharp \colon  \Lint^0(\mathit{T}^\ast \Sp) \rightarrow \Lint^0(\mathit{T} \Sp).
\end{equation}
The {\bf gradient} of a function $f\in\W_\loc(\Sp)$ is defined as $\nabla f:=(\d f)^\sharp\in L^2_\loc(T\Sp)$.

We say that $f\in  \Lint^2_\loc(\Sp)$ has {\bf Laplacian} in $L^2_\loc$, and write $f \in \D(\Delta_\loc)$ if there exists $h \in \Lint^2_\loc(\Sp)$ such that for every $g \in \W(\Sp)$ with bounded support it holds
\[ 
\displaystyle \int g  h \,\dint \mea = - \int \la\nabla f,  \nabla g \ra\dint \mea.  
\]
In this case we put $\Delta f:=h$. 
If $f,h\in L^2(\Sp)$ we shall write $f\in D(\Delta)$ instead of $f\in D(\Delta_\loc)$. It is not hard to check that in this case this notion is equivalent to the more familiar notion of Laplacian as infinitesimal generator of the Dirichlet form 
\[
\ener(f):=\left\{\begin{array}{ll}
\dfrac12\displaystyle{\int|\d f|^2\,\d\mm}&\qquad\text{ if }f\in \W(\Sp),\\
+\infty&\qquad\text{ otherwise}.
\end{array}
\right.
\]
(see \cite{AmbrosioGigliSavare11}, \cite{AmbrosioGigliSavare11-2}).

\bigskip

To continue this introduction, we shall now assume that $\mms$ is a $\RCD(K,\infty)$ space for some $K\in\R$.  A relevant property of Sobolev functions in relations to the metric of such spaces is the following result, proved in \cite{AmbrosioGigliSavare11-2} (see also \cite{Gigli13}, \cite{Gigli13over} for the given formulation):
\begin{theorem}\label{thm:isorcd}
Let $\mmso$ and $\mmst$ be two $\RCD(K,\infty)$ spaces with $\mm_1,\mm_2$ having full support and $T:\Sp_1\to\Sp_2$ and $S:\Sp_2\to\Sp_1$ be Borel maps such that
\[
T\circ S=\Id_{\Sp_2}\quad\mm_2-a.e.\qquad\qquad S\circ T=\Id_{\Sp_1}\quad\mm_1-a.e.,
\]
and
\[
T_*\mm_1=\mm_2\qquad\qquad \ener_{\Sp_1}(f\circ T)=\ener_{\Sp_2}(f)\quad\forall f\in L^2(\Sp_2).
\]
Then, up to modifications in a $\mm_1$-negligible set, $T$ is an isometry.
\end{theorem}

Recall that the class of {\bf test functions} (see \cite{Savare13}) is defined as
\[
\test\X:=\Big\{f\in L^\infty\cap W^{1,2}(\Sp)\cap D(\Delta)\ :\ |\d f|\in L^\infty(\Sp),\ \Delta f\in W^{1,2}(\Sp)\Big\}.
\]
Crucial properties of test functions are that they form an algebra and that $|\d f|^2\in W^{1,2}(\Sp)$ for $f\in \test\Sp$  (see \cite{Savare13}). For our discussion, it is also useful to keep in mind that
\begin{equation}
\label{eq:testsupp}
\begin{split}
&\text{for every $K\subset\Omega\subset \Sp$ with $\sfd(x,y)\geq c$ for some $c>0$ and every $x\in K$, $y\in\Omega^c$}\\
&\text{there exists $f\in\test\X$ with $\supp(f)\subset\Omega$ and identically 1 on $K$},
\end{split}
\end{equation}
see e.g.\  \cite{AmbrosioMondinoSavare13-2}.

With this said, we can define the space $W^{2,2}_\loc(\Sp)$ as the space of $f\in W^{1,2}_\loc(\Sp)$ for which there exists $A\in L^2_\loc((T^*)^{\otimes 2}\Sp)$ such that
\begin{equation}
\label{eq:defhess}
\begin{split}
2\int& hA(\nabla g,\nabla \tilde g)\,\d\mm\\
&=-\int\la\nabla f,\nabla g\ra\div(h\nabla\tilde g)+\la\nabla f,\nabla \tilde g\ra\div(h\nabla g)+h\la\nabla f,\nabla(\la\nabla g,\nabla\tilde g\ra)\ra\,\d\mm
\end{split}
\end{equation}
for every $g,\tilde g,h\in\test\X$ with bounded support. In this case we call $A$ the {\bf Hessian} of $f$ and denote it by  $\H f$. If $f\in \W(\Sp)$ and $\H f\in  L^2((T^*)^{\otimes 2}\Sp)$ we say that $f\in W^{2,2}(\Sp)$. Noticing that the two sides of \eqref{eq:defhess} are continuous in $h$ w.r.t.\ the $W^{1,2}$-norm and using property \eqref{eq:testsupp} it is easy to check that this definition of $W^{2,2}(\Sp)$ coincides with the one given in \cite{Gigli14}. Also, $W^{2,2}(\Sp)$ is a separable Hilbert space when endowed the norm
\[
\|f\|_{W^{2,2}(\Sp)}^2:=\|f\|_{L^2(\Sp)}^2+\||\D f|\|_{L^2(\Sp)}^2+\||\H f|_\HS\|_{L^2(\Sp)}^2.
\]

We recall (see \cite{Gigli14}) that $D(\Delta)\subset W^{2,2}(\Sp)$ and
\begin{equation}
\label{eq:boundhess}
\int|\H f|_\HS^2\,\d\mm\leq \int |\Delta f|^2-K|\d f|^2\,\d\mm,
\end{equation}
so that in particular $\test\Sp\subset W^{2,2}(\Sp)$. We then define $H^{2,2}(\Sp)$ as the $W^{2,2}$-closure of $\test\Sp$ and similarly$H^{2,2}_\loc(\Sp)$ as the $W^{2,2}_\loc(\Sp)$-closure of $\test\Sp$, i.e.: $f\in H^{2,2}_\loc(\Sp)\subset W^{2,2}_\loc(\Sp)$ provided there exists a sequence $(f_n)\subset \test\Sp$ such that $f_n,\d f_n,\H {f_n}$ converge to $f,\d f,\H f$ in $L^2_\loc(\Sp),L^2_\loc(T^*\Sp),L^2_\loc((T^*)^{\otimes 2}\Sp)$ respectively. 

\bigskip

The space of Sobolev vector fields $W^{1,2}_{C,\loc}(T\Sp)$ is defined as the space of $X\in L^2_\loc(T\Sp)$ for which there is $T\in L^2(T^{\otimes 2}\Sp)$ such that
\[
\int h T(\nabla g,\nabla\tilde g)\,\d\mm=\int -\la X,\nabla\tilde g\ra\div(h\nabla g) +h\H {\tilde g}(X,\nabla g)\,\d\mm
\]
for every $h,g,\tilde g\in\test\Sp$ with bounded support. In this case we call $T$ the {\bf covariant derivative} of $X$ and denote it as $\nabla X$. If $|X|,|\nabla X|_\HS\in L^2(\Sp)$ we shall say that $X\in W^{1,2}_C(T\Sp)$; again, it is not hard to check that this definition of $W^{1,2}_C(T\Sp)$ coincides with the one given in \cite{Gigli14}. Vector fields of the form $g\nabla f$ for $f,g\in\test\X$ are in $W^{1,2}_C(T\X)$ and the $W^{1,2}_C$-closure of the linear span of such vector fields will be denoted  $H^{1,2}_{C}(T\Sp)$. The space $H^{1,2}_{C,\loc}(T\X)$ is then equivalently defined either as the subspace of $L^2_\loc(T\X)$ made of vectors $X$ of such that $fX\in H^{1,2}_C(T\X)$ for every $f\in\test\X$ with bounded support or as the $W^{1,2}_{C,\loc}$-closure of $H^{1,2}_C(T\X)$, i.e.\ as the space of vector fields $X\in W^{1,2}_{C,\loc}(T\X)$ such that there is $(X_n)\subset H^{1,2}_C(T\X)$ such that  $X_n\to X$ and $\nabla X_n\to\nabla X$ in $L^2_\loc(T\X)$ and $L^2_\loc(T^{\otimes 2}\X)$ as $n\to\infty$.

The `local versions' of the exterior differential,  codifferential and  Hodge Laplacian are defined in the same way. For us it will be relevant to know that
\begin{equation}
\label{eq:lhodge}
f\in D(\Delta_\loc)\quad\Leftrightarrow\quad \d f\in D(\Delta_{H,\loc})\qquad\text{and in this case }\qquad\d\Delta f=-\Delta_{H}\d f 
\end{equation}
where the minus sign is due to the usual sign convention $\Delta f=-\Delta_Hf$; this is a direct consequence of the analogous identity valid in for objects in $D(\Delta),D(\Delta_H)$ and a cut-off argument. Also, we shall use the fact that if $X^\flat\in D(\Delta_{H,\loc})$ then $X\in W^{1,2}_\loc(T\X)$ and the Bochner inequality
\begin{equation}
\label{eq:bocgen}
\Delta\frac{|X|^2}2\geq |\nabla X|_\HS^2-\langle X^\flat,\Delta_HX^\flat\rangle+K|X|^2
\end{equation}
holds in the weak form, i.e.:
\[
\int \frac{|X|^2}2\,\Delta g\,\d\mm\geq \int g\big(|\nabla X|_\HS^2-\langle X^\flat,\Delta_HX^\flat\rangle+K|X|^2\big)\,\d\mm
\]
for every $g\in\test\X$ non-negative and with bounded support. As before, this follows with a cut-off argument starting with the analogous inequalities valid for the various objects in $L^2$. We remark that \eqref{eq:lhodge} for a vector field $X\in D(\Delta_H)$ implies, in particular, by integration that
\begin{equation}
\label{eq:energiah}
\int |\nabla X|_\HS^2\,\d\mm\leq \int \langle X^\flat,\Delta_HX^\flat\rangle-K|X|^2\,\d\mm
\end{equation}
and recall that
\begin{equation}
\label{eq:hlapdiv}
X^\flat\in D(\Delta_{H})\text{ with }\Delta_{H}X=0\quad\Leftrightarrow \quad\left\{\begin{array}{l}
X^\flat\in D(\d),\ X\in D(\div)\\
\text{ with }\d X^\flat =0,\ \div X=0
\end{array}
\right.
\end{equation}
indeed $\Leftarrow$ follows from the definition of $\Delta_H$ and $\Rightarrow$ from the identity
\[
\int \langle X^\flat,\Delta_HX^\flat\rangle\,\d\mm=\int |\d X^\flat|^2+|\div X|^2\,\d\mm.
\]

\subsection{Regular Lagrangian Flows and continuity equation}
Here we recall the concept of Regular Lagrangian Flow of Sobolev vector fields on $\RCD$ spaces, which provides a non-smooth analogous of the concept of solution of the ODE
\[
\gamma_t'=X_t(\gamma_t).
\]
This notion has been introduced by Ambrosio (\cite{Ambrosio04}) in the Euclidean space in the context of Di Perna-Lions theory (\cite{DiPerna-Lions89}). Then Ambrosio-Trevisan (\cite{Ambrosio-Trevisan14}, see also \cite{AT15}) showed that theory could be developed in $\RCD$ spaces. 
\begin{definition}[Regular Lagrangian flow]\label{def:RLF}
Let $(X_t)\in L^2([0,1],L^2_\loc(T\X))$. We say that $\Fl^{(X_t)}:[0,1]\times\X\to \X$ is a Regular Lagrangian Flow for $(X_t)$ provided:
\begin{itemize}
\item[i)] There is $C>0$ such that
\begin{equation}
\label{eq:bdflow}
(\Fl^{(X_t)}_s)_*\mm\leq C\mm\qquad\forall s\in[0,1].
\end{equation}
\item[ii)] For $\mm$-a.e.\ $x\in\X$ the curve $[0,1]\ni s\mapsto \Fl^{(X_t)}_s(x)\in \X$ is continuous and such that $\Fl^{(X_t)}_0(x)=x$.
\item[iii)] for every $f\in W^{1,2}(\X)$ we have: for $\mm$-a.e.\ $x\in\X$ the function $s\mapsto f(\Fl^{(X_t)}_s(x))$ belongs to $W^{1,1}(0,1)$ and it holds
\begin{equation}
\label{eq:deffl}
\frac{\d}{\d s}f(\Fl^{(X_t)}_s(x))=\d f(X_s)(\Fl^{(X_t)}_s(x))\qquad \mm\times\mathcal L^1\restr{[0,1]}\ae (x,s).
\end{equation}
\end{itemize}
\end{definition}
We shall also deal with Regular Lagrangian Flows for  $X_t\equiv X$: in this case, we shall mainly consider flows defined on the whole $[0,\infty)$ rather than only on $[0,1]$ and the various statements below should be read with this implicit assumption in mind.

\medskip

Notice that it is due to property $(i)$ that property $(iii)$ makes sense. Indeed, for given $X_s\in L^2(T\X)$ and $f\in W^{1,2}(\X)$ the function $\d f(X_s)\in L^1(\X)$ is only defined $\mm$-a.e., so that (part of) the role of  \eqref{eq:bdflow} is to grant that $\d f(X_s)\circ F_s$ is  well defined $\mm$-a.e.. Also, it is known  (see Lemmas 7.4 and 9.2 in \cite{Ambrosio-Trevisan14}) that  for $\mm$-a.e.\ $x$ the curve $s\mapsto \Fl^{(X_t)}_s(x) $ is absolutely continuous and for its metric speed $|\dot\Fl^{(X_t)}_s(x)|$ we have
\begin{equation}
\label{eq:metrsp}
|\dot\Fl^{(X_t)}_s(x)|=|X_s|(\Fl^{(X_t)}_s(x))\qquad a.e.\ s\in[0,1].
\end{equation}
We shall always work with vector fields $(X_t)$ such that 
\begin{equation}
\label{eq:linftyspeed}
|X_t|\in L^\infty([0,1],L^\infty(\X)).
\end{equation} 
In this case, a simple property, valid for any $p\in[1,\infty)$, of Regular Lagrangian Flows that we shall occasionally use is the following:
\begin{equation}
\label{eq:contl2loc}
f_s\to f\quad\text{ in $\Lint^p(\Sp)$ as $s\to 0$}\qquad\Rightarrow\qquad  f_s\circ\Fl^{(X_t)}_s\to f\quad\text{ in $\Lint^p(\Sp)$ as $s\to 0$}.
\end{equation}
Indeed, 
for any Lipschitz function $\tilde f$ with bounded support we have
\[
\begin{split}
\| f_s\circ\Fl^{(X_t)}_{s}-f\|_{L^p}&\leq\| f_s\circ\Fl^{(X_t)}_{s}-\tilde f\circ \Fl^{(X_t)}_{s}\|_{L^p}+\| \tilde t\circ\Fl^{(X_t)}_{s}-\tilde f\|_{L^p}+\| \tilde f -f\|_{L^p}\\
\text{by \eqref{eq:bdflow}}\qquad\qquad&\leq (C^{1/p}+1)\| f_t-\tilde f\|_{L^p}+\| \tilde f\circ\Fl^{(X_t)}_{t}-\tilde f\|_{L^p}.
\end{split}
\]
Since for $\mm$-a.e.\ $x\in\Sp$ the curve $s\mapsto \Fl^{(X_t)}_s(x)$ is Lipschitz and with metric speed bounded above by $ \||X_t|\|_{L^\infty_t(L^\infty_x)}$, we have  
\[
|\tilde f(\Fl^X_{t}(x))-\tilde f(x)|\leq |t|\Lip(\tilde f)\||X_t|\|_{L^\infty_t(L^\infty_x)}\qquad\mm-a.e.\ x\in\Sp,
\]
hence  letting $t\to 0$ in the above we obtain
\[
\lims_{t\to 0}\| f_t\circ\Fl^X_{t}-f\|_{L^2}\leq ( C^{1/p}+1)\| \tilde f -f\|_{L^p},
\]
so that \eqref{eq:contl2loc} follows from the arbitrariness of $\tilde f$ and the density of Lipschitz functions with bounded support in $L^p(\Sp)$.

We shall mainly use Regular Lagrangian Flows via the following characterization:
\begin{proposition}\label{prop:rlfeq}
Let $(X_t)\in L^2([0,1],L^2_\loc(T\X))$ be such that \eqref{eq:linftyspeed} holds and $F:[0,1]\times\X\to \X$ be a Borel map satisfying $(i),(ii)$ of Definition \ref{def:RLF}. Then the following are equivalent:
\begin{itemize}
\item[a)] (iii) of Definition \ref{def:RLF} holds, i.e.\ $F$ is a Regular Lagrangian flow for $(X_t)$.
\item[b)] for every $f\in W^{1,2}(\X)$ the map $[0,1]\ni t\mapsto f\circ F_t\in L^2(\X)$ is Lipschitz and for a.e.\ $t\in[0,1]$ it holds
\begin{equation}
\label{eq:derl2}
\lim_{h\to 0}\frac{f\circ F_{t+h}-f\circ F_t}{h}=\d f(X_t)\circ F_t,
\end{equation}
the limit being intended in $L^2(\X)$.
\item[c)] for every $f\in W^{1,2}_\loc(\X)$ the map $[0,1]\ni t\mapsto f\circ F_t\in L^2_\loc(\X)$ is Lipschitz and for a.e.\ $t\in[0,1]$ \eqref{eq:derl2} holds with the limit being intended in $L^2_\loc$.
\end{itemize}
Moreover, if these holds and $X_t\equiv X$, then `Lipschitz' in $(b),(c)$ can be replaced by `$C^1$' and \eqref{eq:derl2} holds for every $t\in[0,1]$.
\end{proposition}
\begin{proof}\\
${\mathbf {(a)\Rightarrow (b)}}$  From \eqref{eq:deffl} and Fubini's theorem we have that for $\mm$-a.e.\ $x$ and  $(\mathcal L^1\restr{[0,1]})^2$-a.e.\ $(s_0,s_1)$ it holds
\begin{equation}
\label{eq:punt}
f(\Fl^{(X_t)}_{s_1}(x))-f(\Fl^{(X_t)}_{s_0}(x))=\int_{s_0}^{s_1}\d f(X_s)(\Fl^{(X_t)}_{s}(x)\,\d s\,.
\end{equation}
By the uniform bound \eqref{eq:linftyspeed}  and \eqref{eq:bdflow} we deduce that $(\d f(X_\cdot)\circ\Fl^{(X_t)}_\cdot)\in L^\infty([0,1],L^2(\Sp))$, and thus the Bochner integral $\int_{s_0}^{s_1}\d f(X_s)\circ\Fl^{(X_t)}_{s}\,\d s$ is a well defined function in $L^2(\Sp)$ which vary continuously in $s_0,s_1$. By  \eqref{eq:contl2loc} we also deduce that $s\mapsto f\circ \Fl^{(X_t)}_s\in L^2(\Sp)$ is continuous, thus from \eqref{eq:punt} we obtain that
\[
f\circ \Fl^{(X_t)}_{s_1}-f\circ \Fl^{(X_t)}_{s_0}=\int_{s_0}^{s_1}\d f(X_s)\circ\Fl^{(X_t)}_{s}\,\d s,\qquad\forall s_0,s_1\in[0,1],\ s_0<s_1,
\]
where the identity is intended in $L^2(\Sp)$ and the integral in the right hand side is the Bochner one. The Lipschitz continuity of $s\mapsto f\circ\Fl^{(X_t)}_s\in L^2(\Sp)$ and \eqref{eq:derl2} follow.\\
${\mathbf {(b)\Rightarrow (c)}}$ The assumption \eqref{eq:linftyspeed} together with \eqref{eq:metrsp} grant finite speed of propagation. Then the claim follows by a simple cut-off argument.\\
${\mathbf {(c)\Rightarrow (a)}}$ By assumption for every bounded set $B\subset \X$ we have
\[
\nchi_B\Big(f\circ \Fl^{(X_t)}_{s_1}-f\circ \Fl^{(X_t)}_{s_0}\Big)=\int_{s_0}^{s_1}\nchi_B\d f(X_s)\circ\Fl^{(X_t)}_{s}\,\d s,\qquad\forall s_0,s_1\in[0,1],\ s_0<s_1
\]
and thus from the arbitrariness of $B$ and Fubini's theorem we conclude that for $\mm$-a.e.\ $x$ it holds
\[
f( \Fl^{(X_t)}_{s_1}(x))-f( \Fl^{(X_t)}_{s_0}(x))=\int_{s_0}^{s_1}\d f(X_s)(\Fl^{(X_t)}_{s}(x))\,\d s,\qquad\mathcal L^2-a.e.\ s_0,s_1\in[0,1],\ s_0<s_1.
\]
Applying Lemma 2.1 of \cite{AmbrosioGigliSavare11-3} we deduce that for $\mm$-a.e.\ $x$ the function $t\mapsto f( \Fl^{(X_t)}_{s_1}(x))$ belongs to $W^{1,1}(0,1)$ and it is now obvious that its distributional derivative is given by $\d f(X_s)(\Fl^{(X_t)}_{s}(x))$, thus concluding the proof.\\
${\mathbf {C^1\ regularity}}$ It is sufficient to prove that $s\mapsto \d f(X)\circ\Fl^{(X)}_s\in L^2(\Sp)$ (resp. $L^2_\loc$) is continuous for $f\in W^{1,2}(\X)$ (resp.\ $W^{1,2}_\loc(\X)$). This is a direct consequence of \eqref{eq:contl2loc} applied to the functions $f_s=\d f(X)$ (resp.\ $\nchi_B\d f(X)$ for $B\subset \X$ Borel and bounded).
\end{proof}

With this said, we shall now recall the  main result of \cite{Ambrosio-Trevisan14} in the form we will use it:
\begin{theorem}\label{thm:AT}
Let $(X_t)\in L^2([0,1],W^{1,2}_{C,\loc}(T\X))\cap  L^\infty([0,1],L^\infty(T\X)) $ be such that $X_t\in D(\div_\loc)$ for a.e.\ $t\in[0,1]$, with 
\begin{equation}
\label{eq:iprlf}
\int_0^1\||\nabla X|\|_{L^2(\Sp)}+\|\div(X_t)\|_{L^2(\X)}+\|\big(\div(X_t)\big)^-\|_{L^\infty(\X)}\,\d t<\infty.
\end{equation}
Then a Regular Lagrangian Flow $F_s^{(X_t)}$ for $(X_t)$ exists and is unique, in the sense that if $\tilde F^{(X_t)}$ is another flow, then for $\mm$-a.e.\ $x\in \X$ it holds $F_s(x)=\tilde F_s(x)$ for every $s\in[0,1]$. Moreover it holds the quantitative bound
\begin{equation}
\label{eq:quantbound}
(F_s^{(X_t)})_*\mm\leq \exp\Big(\int_0^s\|\big(\div(X_t)\big)^-\|_{L^\infty(\X)}\,\d t\Big)\,\mm\qquad\forall s\in[0,1].
\end{equation}
\end{theorem}
Notice that the uniqueness part of the statement grants in particular that for $X_t$ independent on $t$, say $X_t= X$, the Regular Lagrangian Flow (that in this situation is well defined for any $t\geq 0$) satisfies the semigroup property
\begin{equation}
\label{eq:semigrpr}
\Fl^{(X)}_t\circ\Fl^{(X)}_s=\Fl^{(X)}_{t+s}\quad\mm-a.e.\qquad\forall t,s\geq 0.
\end{equation}
The uniqueness part of Theorem \ref{thm:AT} is tightly linked to the uniqueness of solutions of the continuity equation (\cite{Ambrosio-Trevisan14},\cite{GigliHan13}):
\begin{definition}[Solutions of the continuity equation]\label{def:solce}
Let $t\mapsto\mu_t\in \prob\Sp$ and $t\mapsto X_t\in L^0(T\Sp)$, $t\in[0,1]$,  be Borel maps. We say that they solve the continuity equation
\begin{equation}
\label{eq:ce}
\frac\d{\d t}\mu_t+\div(X_t\mu_t)=0
\end{equation}
provided:
\begin{itemize}
\item[i)] $\mu_t\leq C\mm$ for every $t\in[0,1]$ and some $C>0$,
\item[ii)] we have
\[
\int_0^1\int |X_t|^2\,\d\mu_t\,\d t<\infty,
\]
\item[iii)] for any $f\in W^{1,2}(\Sp)$ the map $t\mapsto\int f\,\d\mu_t$ is absolutely continuous and it holds
\[
\frac\d{\d t}\int f\,\d\mu_t=\int \d f(X_t)\,\d\mu_t\qquad a.e.\ t.
\]
\end{itemize}
\end{definition}
The following uniqueness result and representation formula has been proved in \cite{Ambrosio-Trevisan14}.
\begin{theorem}[Uniqueness of solutions of the continuity equation]\label{thm:unice}
Let $(X_t)$ be as in Theorem \ref{thm:AT} and $\bar\mu\in\prob\Sp$ be such that $\mu_0\leq C\mm$ for some $C>0$.

Then there exists a unique  $(\mu_t)$ such that $(\mu_t,X_t)$ solves  the continuity equation \eqref{eq:ce} in the sense of Definition \ref{def:solce} and for which $\mu_0=\bar\mu$. Moreover, such $(\mu_t)$ is given by
\[
\mu_s=(\Fl^{(X_t)}_s)_*\bar\mu\qquad\forall s\in[0,1].
\]
\end{theorem}

\section{Calculus Tools}

\subsection{Local bounded compression/deformation}
Here we shall extend the constructions of pullback module and pullback of 1-forms provided in \cite{Gigli14} (see also \cite{Gigli17}) to maps which are locally of bounded compression/deformation. For this, it is technically convenient  to work with $L^0$-normed modules rather than with $L^2$-normed ones.

\subsubsection{Pullback module through a map of local bounded compression}

\begin{definition}[Maps of local bounded compression]
Let $(\Sp_1, \mea_1)$ and $(\Sp_2, \mea_2)$ be two $\sigma$-finite measured spaces. We say that  $\varphi \colon \Sp_1 \rightarrow \Sp_2$ is a map of   local bounded compression provided  for every $B \subset \Bor(\Sp_1)$ with $\mea_1(B) < +\infty$ there exists a constant $C_B \ge 0$ such that
\begin{equation}
\varphi_{\ast} \big( \mea_1\restr{B}  \big) \le C_B \mea_2.
\end{equation}
\end{definition}

\begin{thmdef}\label{pbmod}
Let $(\Sp_1, \mea_1)$ and $(\Sp_2, \mea_2)$ be two $\sigma$-finite measured spaces, $\varphi \colon \Sp_1 \rightarrow \Sp_2$ be of   local bounded compression and $\mathscr{M}$  a $\Lint^0(\Sp_2)$-normed module. Then there exists a unique couple $\left( [\varphi^\ast] \mathscr{M}, \left[ \varphi^\ast \right] \right)$ where $[\varphi^\ast] \mathscr{M}$ is a $\Lint^0(\Sp_1)$-module and $\left[ \varphi^\ast \right] \colon \mathscr{M} \rightarrow [\varphi^\ast] \mathscr{M}$ is a linear map such that:
\begin{enumerate}
\item[1)] $\abs*{\left[ \varphi^\ast \right] v} = \abs*{v} \circ \varphi$ $\mea_1$-a.e.\ for every $v\in\mathscr M$,
\item[2)] $[\varphi^\ast] \mathscr{M}$ is generated by  $\left\{ \left[ \varphi^\ast \right] (v) : v \in \mathscr{M}  \right\}$.
\end{enumerate} 
Uniqueness is intended up to unique isomorphism, i.e.:  if  $\big( \widetilde{[\varphi^\ast] \mathscr{M}}, \big[ \tilde{\varphi}^\ast \big] \big)$ is another such couple,   then there exists a unique isomorphism $\Phi \colon [\varphi^\ast ]\mathscr{M} \rightarrow \widetilde{[\varphi^\ast ]\mathscr{M}}$ such that $\left[ \tilde{\varphi}^\ast \right] = \left[ \varphi^\ast \right] \circ \Phi$.
\end{thmdef}

\begin{proof}\\ 
{\bf Existence} We define the 'pre-pullback' set {\sf Ppb} as
\begin{equation} \notag
\begin{split}
\text{Ppb} := \big\{ (v_i, A_i)_{i=1,\ldots,n} \,: \, n\in\N,\  (A_i)\subset \Bor(\Sp_1)  \text{ is a  partition of $\Sp_1$}\ \text{and}\,\,  v_i \in \mathscr{M} \,\, \forall i =1,\ldots,n \big\}
\end{split}
\end{equation}
and an equivalence relation on it by declaring $(v_i, A_i)  \sim (w_j, B_j)$ provided
\[  
\abs*{v_i - w_j} \circ \varphi = 0 \quad \mea_1\text{-a.e. on}\,\,  A_i \cap B_j \qquad \forall \, i,j.  
 \]
 It s readily verified  that it is actually an equivalence relation on {\sf Ppb}:  we shall denote by $[ (v_i, A_i) ]$ the equivalence class of $(v_i, A_i)$.
 
We endow {\sf Ppb}$/\sim$ with a vector space structure by putting
\[
\begin{split}
[(v_i, A_i)] + [(w_j, B_j) ] &:= [ (v_i + w_j, A_i \cap B_j) ]\\
 \lambda [ (v_i, A_i) ]& := [ (\lambda v_i, A_i) ]
\end{split}
\] 
for every $[(v_i, A_i)] , [(w_j, B_j) ]\in{\sf Ppb}/\sim $ and $\lambda\in\R$. Notice that these are well defined. Moreover, we define a pointwise norm on ${\sf Ppb}/\sim$ and a multiplication with simple functions by putting
\[
\begin{split}
| [ (v_i, A_i)  ]|&:=\sum_i\nchi_{A_i}|v_i|\circ\varphi\\
g [ (v_i, A_i)  ] &:=  [ (\alpha_j v_i, A_i \cap E_j) ]\qquad\text{for}\qquad g=\sum_j\alpha_j\nchi_{E_j}.
\end{split}
\]
Again, these are easily seen to be well defined; then  we fix a partition $(E_i)\subset\Bor(\Sp_1)$ of $\Sp_1$ made of sets of finite $\mea_1$-measure and define the distance $\dint_0$ on ${\sf Ppb}/\sim$ as
\[
 \dint_0 ([ (v_i, A_i)], [ (w_j, B_j) ]) := \displaystyle \sum_{ k \in \N} \dfrac{1}{2^k \mea_1(E_k)} \int_{E_k }\big| [(v_i, A_i)] - [(w_j, B_j) ] \big|\,\dint \mea_1. 
 \]
We then define the space $[\varphi^\ast] \mathscr{M}$ as the completion of $(\text{Ppb}/\sim, \dist_0)$, equipped with the induced topology and the pullback map $[\varphi^*]:\mathscr M\to [\varphi^\ast] \mathscr{M}$ as $[\varphi^*]v:=[v,\Sp_1]$. The (in)equalities
\[
\begin{split}
\abs*{[ (v_i + w_j, A_i \cap B_j)]} &\le \abs*{[ (v_i, A_i) ]} + \abs*{[ (w_j, B_j)]}, \\
\abs*{\lambda [ (v_i, A_i) ]} &= \abs*{\lambda} \abs*{[ (v_i, A_i) ]},\\
\abs*{g [ (v_i, A_i) ]} &= \abs*{g} \abs*{[ (v_i, A_i) ]},\\
\end{split}
\]
valid $\mea_1$-a.e.\ for every $[ (v_i, A_i) ], [ (w_j, B_j)] \in \text{Ppb}/\sim$, $\lambda \in \real$ and simple function $g$ grant that the vector space structure, the pointwise norm and the multiplication by simple functions can all be extended by continuity to the whole $[\varphi^\ast] \mathscr{M}$ and it is then clear with these operations such space is a $L^0$-normed module.

Property $(1)$ then follows by the very definitions of pullback map and pointwise norm, while property $(2)$ from the fact that ${\sf  Ppb}/\sim$ is dense in $[\varphi^\ast] \mathscr{M}$ and the typical element $[v_i,A_i]$ of ${\sf Ppb}/\sim$ is equal to $\sum_i\nchi_{A_i}[\varphi^*]v_i$.

\noindent {\bf Uniqueness} The requirement for $\Phi \colon [\varphi^\ast] \mathscr{M} \rightarrow \widetilde{[\varphi^\ast ]\mathscr{M}}$ to be $\Lint^0(\Sp_1)$-linear and such that $\left[ \tilde{\varphi}^\ast \right] = \left[ \varphi^\ast \right] \circ \Phi$ force the definition
\begin{equation} \label{DefPhi} 
\Phi(V) := \sum_i\nchi_{A_i}[ \tilde{\varphi}^*]v_i \qquad\text{for}\qquad V=\sum_i\nchi_{A_i}[\varphi^*]v_i.
\end{equation}
The identity
\[
|\Phi(V)|=\sum_i\nchi_{A_i}|[ \tilde{\varphi}^*]v_i |\stackrel{\text{(1) for $\widetilde{[\varphi^\ast] \mathscr{M}}$}}=\sum_i\nchi_{A_i}|v_i|\circ\varphi\stackrel{\text{(1) for ${[\varphi^\ast ]\mathscr{M}}$}}=\sum_i\nchi_{A_i}|[ {\varphi}^*]v_i |=|V|
\]
shows in particular that the definition of $\Phi(V)$ is well-posed, i.e.\ it depends only on $V$ and not on the particular way to represent it as sum. It also shows that it preserves the pointwise norm and thus it is continuous. Since the space of $V$'s of the form $\sum_i\nchi_{A_i}[\varphi^*]v_i$ is dense in $[\varphi^\ast] \mathscr{M} $ (property $(2)$ for $[\varphi^\ast] \mathscr{M} $), such $\Phi$ can be uniquely extended to a continuous map on the whole $[\varphi^\ast ]\mathscr{M} $ and such extension is clearly linear, continuous, and preserves the pointwise norm. By definition, it also holds $\Phi(gV)=g\Phi(V)$ for $g$ simple and $V=\sum_i\nchi_{A_i}[\varphi^*]v_i$ and thus by approximation we see that the same holds for general $g\in L^0(\Sp_1)$ and $V\in[ \varphi^\ast ]\mathscr{M}$. 

It remains to show that the image of $\Phi$ is the whole $\widetilde{[\varphi^\ast] \mathscr{M}}$: this follows from the fact that elements of the form $\sum_i\nchi_{A_i}[ \tilde{\varphi}^*]v_i$, which  by definition are in the image of $\Phi$, are dense in $\widetilde{[\varphi^\ast ]\mathscr{M}}$ by property $(2)$ for $\widetilde{[\varphi^\ast] \mathscr{M}}$.
\end{proof}
We shall now provide an explicit representation of such pullback module in the  case when $\varphi$ is a projection.

Thus  let  $(\Sp_1, \mea_1)$ and $(\Sp_2, \mea_2)$ be two $\sigma$-finite measured spaces and let $\mathscr{M}$ be a $\Lint^0(\Sp_1)$-module over $\Sp_1$.

We shall denote by  $\Lint^0(\Sp_2, \mathscr{M})$ the space of (equivalence classes up to $\mm_2$-a.e.\ equality of)  strongly measurable (=Borel and essentially separably valued)  functions from $\Sp_2$ to $\mathscr{M}$ and claim that such space canonically carries the structure of $L^0(\Sp_1\times\Sp_2)$-normed module. The multiplication of an element in $\Lint^0(\Sp_2, \mathscr{M})$ by a function $f \in \Lint^0(\Sp_1\times\Sp_2)$ is defined as the map $\Sp_2 \ni x_2 \mapsto f(\cdot, x_2) \cdot v(\cdot , x_2) \in \mathscr{M}$. By approximating $f$ in $L^0(\Sp_1\times\Sp_2)$ with functions having finite range it is easy to check that $x_2 \mapsto f(\cdot, x_2) \cdot v(\cdot , x_2)$ has separable range if $x_2 \mapsto  v(\cdot , x_2)$ does, so that this definition is well posed. Similarly, the pointwise norm of  $v \in \Lint^0(\Sp_2, \mathscr{M})$ is obtained composing the map  $x_2 \mapsto v(\cdot, x_2) \in \mathscr{M}$ with the pointwise norm on $\M$, thus providing an element of $L^0(\Sp_2,L^0(\Sp_1))\sim L^0(\Sp_1\times\Sp_2)$. Finally, we use such pointwise norm to define the topology of $\Lint^0(\Sp_2, \mathscr{M})$ as in point $(iii)$ of Definition \ref{def:l0mod}. 

It is not hard to check  that sequences converging in this topology  are made of maps $v_n(\cdot,x_2)$ which are $\mm_2$-a.e.\ converging and that with these definitions $\Lint^0(\Sp_2, \mathscr{M})$ is indeed a $L^0(\Sp_1\times\Sp_2)$-normed module.

In what will come next, we shall often implicitly use the following identification:
\begin{proposition}[$ {\sqb{\pi_{1}}^{\ast}} \mathscr{M}$ is isomorphic to  $\Lint^0(\Sp_2, \mathscr{M})$]\label{isopull}
Let $(\Sp_1, \mea_1)$ and $(\Sp_2, \mea_2)$ be two $\sigma$-finite measured spaces, $\mathscr{M}$ a $\Lint^0(\Sp_1)$-module over $\Sp_1$ and $\pi_1:\Sp_1\times\Sp_2\to\Sp_1$ the canonical projection.

Then there exists a unique isomorphism from $[{\pi_{1}}^{\ast}] (\mathscr{M}) $ to $\Lint^0(\Sp_2, \mathscr{M})$ which for every $v\in\M$ sends   $[\pi_1^\ast]v$ to the function in $ \Lint^0(\Sp_2, \mathscr{M})$ constantly equal to $v$.
\end{proposition}

\begin{proof} For $v\in\M$ let $\hat v\in L^0(\Sp_2,\M)$ be the function constantly equal to $v$. It is clear that   $\abs*{ \hat v} = \abs{v} \circ \pi_1$ $\mm_1\times\mm_2$-a.e.. Moreover, from the fact that functions in $ L^0(\Sp_2,\M)$ are essentially separably valued it follows by standard means in vector-space integration that $\{\hat v:v\in\M\}$ generate the whole $ L^0(\Sp_2,\M)$.

The conclusion comes from Theorem \ref{pbmod}.
\end{proof} \bigskip

\subsubsection{Localized pullback of 1-forms}

\begin{definition}[Maps of local bounded deformation]
Let $\mmso$ and $\mmst$ be two metric measure spaces. We say that a map $\varphi : \Sp_1 \rightarrow \Sp_2$ is  of local bounded deformation if for every bounded set $B\subset \Sp_1$ there are constants $L(B),C(B)>0$ such that:
\[
\begin{split}
\varphi&\ \mathrm{ is}\ L(B)\mathrm{-Lipschitz\ on }\ B\\
\varphi_{\ast}& \big( \mea_1\restr B \big) \le C(B) \mea_2.
\end{split}
\]
\end{definition}
Recalling that the local Lipschitz constant $\lip\,\varphi:\Sp_1\to[0,\infty]$ is defined as
\[
\lip\,\varphi(x):=\lims_{y\to x}\frac{\sfd_2(\varphi(x),\varphi(y))}{\sfd_1(x,y)}
\]
if $x$ is not isolated, 0 otherwise, we have  the following simple statement:
\begin{proposition}
Let $\mmso$ and $\mmst$ be two metric measure spaces and $\varphi : \Sp_1 \rightarrow \Sp_2$ be a map of local bounded deformation. Then  for any $f \in \Sclass_\loc(\Sp_2)$, we have $f \circ \varphi \in \Sclass_{\loc}(\Sp_1)$ and
\begin{equation}\label{diffbound}
|\dint (f \circ \varphi)| \le \lip\, \varphi \,|\dint f| \circ \varphi,\qquad\mea_1-a.e..
\end{equation}
\end{proposition}
\begin{proof} Fix a point $\bar x\in\Sp_1$, let $A_n\subset C([0,1],\Sp_1)$ be defined as
\[
A_n:=\big\{\gamma\in C([0,1],\Sp_1) :\ \gamma_t\in B_n(\bar x)\ \forall t\in[0,1]\big\}
\]
and notice that $\cup_nA_n=C([0,1],\Sp_1)$. Now let $\ppi$ be a test plan on $\Sp_1$ and notice that for $n$ sufficiently large the measure $\ppi_n:=\ppi(A_n)^{-1}\ppi\restr{A_n}$ is well defined and a test plan. By construction we have 
\[
(\e_t)_*\varphi_*\ppi_n\leq \ppi(A_n)^{-1}C(B_n(\bar x))C(\ppi)\mm_2\qquad\forall t\in[0,1],
\]
where $C(\ppi)$ is such that $(\e_t)_*\ppi\leq C(\ppi)\mm_1$ for every $t\in[0,1]$, and taking into account the trivial bound
\begin{equation}
\label{eq:spv}
\{\text{metric speed of $t\mapsto \varphi(\gamma_t)$}\}\ \leq \lip \,\varphi(\gamma_t)|\dot\gamma_t|\qquad a.e.\ t
\end{equation}
we also have
\[
\iint_0^1|\dot\gamma_t|^2\,\d t\,\d\varphi_*\ppi_n(\gamma)\leq \ppi(A_n)^{-1} L^2(B_n(\bar x))\iint_0^1|\dot\gamma_t|^2\,\d t\,\d\ppi_n(\gamma).
\]
Hence  $\varphi_*\ppi_n$ is a test plan on $\Sp_2$ and thus for $\varphi_*\ppi_n$-a.e.\ $\tilde\gamma$ we have that $t\mapsto f(\tilde \gamma)$ is in $W^{1,1}(0,1)$ with
\[
|\frac\d{\d t}f(\tilde\gamma_t)|\leq |\dot{\tilde\gamma}_t||\d f|(\tilde\gamma_t).
\]
Recalling \eqref{eq:spv} this means that for  $\ppi_n$-a.e.\ $\gamma$ the map $t\mapsto f(\varphi(\gamma_t))$ belongs to $W^{1,1}(0,1)$ with
\begin{equation}
\label{eq:ppin}
|\frac\d{\d t}f(\varphi(\gamma_t))|\leq \lip\,\varphi(\gamma_t)\,|\dot{\gamma}_t|\,|\d f|(\varphi(\gamma_t)).
\end{equation}
Being this true for every $n\in\N$ sufficiently large, \eqref{eq:ppin} holds also for $\ppi$-a.e.\ $\gamma$ and since $\lip\, \varphi\, |\dint f| \circ \varphi\in L^2_\loc(\Sp_1)$, by the characterization \eqref{eq:sobcurve} of Sobolev functions the proof is completed. 
\end{proof}
\begin{thmdef}[Pullback of 1-forms] Let $\mmso$ and $\mmst$ be two metric measure spaces and $\varphi : \Sp_1 \rightarrow \Sp_2$ be a map of local bounded deformation.

Then there exists a unique linear and continuous map $\varphi^\ast \colon \Lint^0 (\mathit{T}^{\ast}\Sp_2) \rightarrow \Lint^0 (\mathit{T}^{\ast}\Sp_1)$ such that
\begin{equation}\label{pb1}
\begin{split}
\varphi^{\ast} (\dint f) &= \dint (f \circ \varphi), \quad \quad \quad \forall f \in \Sclass_\loc (\Sp_2),\\
\varphi^{\ast} (g \omega) &= g \circ \varphi \,\varphi^{\ast} \omega, \quad \quad \forall g \in \Lint^0 (\Sp_2),\ \omega \in \Lint^0 (\mathit{T}^{\ast} \Sp_2),
\end{split}
\end{equation}
and such map satisfies
\begin{equation} \label{pb2}
|\varphi^\ast \omega|  \le \lip\, \varphi\, |\omega|  \circ \varphi \quad \quad \mea_1 \text{-a.e.},\qquad \forall \omega \in \Lint^0(\mathit{T}^{\ast}\Sp_2).
\end{equation}
\end{thmdef}
\begin{proof} The requirements  \eqref{pb1} force the definition
\begin{equation}
\label{eq:pb4}
\varphi^\ast \omega:= \displaystyle \sum_i \nchi_{\varphi^{-1}(E_i)} \dint (f_i \circ \varphi)\qquad\text{for}\qquad\omega=\sum_i\nchi_{E_i}\d f_i,
\end{equation}
for $(E_i)$ finite Borel partition of $\Sp_2$ and $(f_i)\subset\Sclass_\loc(\Sp_2)$. The bound
\begin{equation}
\label{eq:pb3}
|\varphi^{\ast}\omega|  = \displaystyle \sum_i 	\nchi_{\varphi^{-1}(E_i)} |\dint (f_i \circ \varphi)|  \stackrel{\eqref{diffbound}}\le \lip\,\varphi\, \sum_i \nchi_{E_i} \circ \varphi \, |\dint f_i|  \circ \varphi = \lip\,\varphi \,|\omega|  \circ \varphi,
\end{equation}
grants both that the definition of $\varphi^*\omega$ is well-posed (i.e.\ its value depends only on $\omega$ and not in how it is written as $\sum_i\nchi_{E_i}\d f_i$) and that $\varphi^*$ is continuous from the space of $\omega$'s as in \eqref{eq:pb4}  with the $\Lint^0 (\mathit{T}^{\ast}\Sp_2)$-topology to $\Lint^0 (\mathit{T}^{\ast}\Sp_1)$. Since the class of such $\omega$'s is dense in $\Lint^0 (\mathit{T}^{\ast}\Sp_2)$, we  can be uniquely extend $\varphi^*$ to a continuous map from $\Lint^0 (\mathit{T}^{\ast}\Sp_2)$ to $\Lint^0 (\mathit{T}^{\ast}\Sp_1)$.

The resulting extension satisfies the first in \eqref{pb1} by definition, while \eqref{pb2} comes from \eqref{eq:pb3}. The second in \eqref{pb1} for simple functions $g$ is a direct consequence of the definition \eqref{eq:pb4}, then the general case follows by approximation.\end{proof}

\subsection{Calculus on   product  spaces}\label{se:calcprod}
\subsubsection{Cotangent module and product of spaces}
Let $\mmso$ and $\mmst$ be two metric measure spaces. Aim of this section is to relate the cotangent modules of $\Sp_1,\Sp_2$ to that of the product space $\Sp_1 \times \Sp_2$, which will be always implicitly endowed with the product measure and the distance
\[
(\sfd_1\otimes\sfd_2)^2\big((x_1,x_2),(y_1,y_2)\big):=\sfd_1^2(x_1,y_1)+\sfd_2^2(x_2,y_2).
\]
Let $\pi_i:\Sp_1\times\Sp_2\to \Sp_i$, $i=1,2$ be the canonical projections, observe that they are of local bounded deformation and recall from Proposition \ref{isopull} and the discussion before it that $L^0(\Sp_2,L^0(T^*\Sp_1))\sim [\pi_2^*]L^0(T^*\Sp_1)$ canonically carries the structure of $L^0(\Sp_1\times\Sp_2)$-normed module.

We then start with the following simple and general fact:
\begin{proposition}\label{prop:defPhi}
Let $\mmso$ and $\mmst$ be two metric measure spaces. Then there exists a unique $L^0(\Sp_1\times\Sp_2)$-linear and continuous map $\Phi_1:L^0(\Sp_2,L^0(T^*\Sp_1))\to L^0(T^*(\Sp_1\times  \Sp_2))$ such that
\begin{equation}
\label{eq:reqphi1}
\Phi_1(\widehat{\d g})=\d(g\circ\pi_1)\qquad\forall g\in \Sclass_\loc(\Sp_1),
\end{equation}
where $\widehat{\d g}:\Sp_2\to L^0(T^*\Sp_1)$ is the function identically equal to $\d g$. Such map preserves the pointwise norm.

Similarly, there is a unique $L^0(\Sp_1\times\Sp_2)$-linear and continuous map $\Phi_2:L^0(\Sp_1,L^0(T^*\Sp_2))\to L^0(T^*(\Sp_1\times  \Sp_2))$ such that
\[
\Phi_2(\widehat{\d h})=\d(h\circ\pi_1)\qquad\forall h\in \Sclass_\loc(\Sp_2),
\]
where $\widehat{\d h}:\Sp_1\to L^0(T^*\Sp_2)$ is the function identically equal to $\d h$, and  such map preserves the pointwise norm.
\end{proposition}
\begin{proof} We shall prove the claims for $\Phi_1$, as then the ones for $\Phi_2$ follow by symmetry. The required $L^0$-linearity and \eqref{eq:reqphi1} force the definition
\begin{equation}
\label{eq:defphi1}
\Phi_1(W):=\sum_{i,j}\nchi_{A_i}\circ\pi_1\nchi_{B_j}\circ\pi_2\,\d(g_{i,j}\circ\pi_1) \qquad\text{for}\qquad W=\sum_{i,j}\nchi_{B_j}(\nchi_{A_i}\d g_{i,j})
\end{equation}
where  $(A_i),(B_j)$ are finite Borel partitions of $\Sp_1,\Sp_2$ respectively and $(g_{i,j})\subset \Sclass_\loc(\Sp_1)$. Since $\pi_1$ is 1-Lipschitz we have
\begin{equation}
\label{eq:normPhi1}
|\Phi_1(W)|=\sum_i\nchi_{B_i}\circ\pi_2|\d(g_i\circ\pi_1)|\stackrel{\eqref{diffbound}}\leq\sum_i\nchi_{B_i}\circ\pi_2|\d g_i|\circ\pi_1=|W|
\end{equation}
which shows both that \eqref{eq:defphi1} provides a good definition for $\Phi_1(W)$, in the sense that $\Phi_1(W)$ depends only on $W$ and not on the way we write it as $\sum_{i,j}\nchi_{B_j}(\nchi_{A_i}\d g_{i,j})$, and that it is continuous. The definition also ensures that $\Phi_1(fW)=f\Phi_1(W)$ for $f\in L^0(\Sp_1\times\Sp_2)$ of the form $\sum_{i,j}\alpha_{i,j}\nchi_{A_i}\nchi_{B_j}$ for $(A_i),(B_j)$ finite Borel partitions of $\Sp_1,\Sp_2$ respectively and $(\alpha_{i,j})\subset\R$. 

Since these functions are dense in $L^0(\Sp_1\times\Sp_2)$ and $W$'s as in \eqref{eq:defphi1}  are dense in $L^0(\Sp_1,L^0(T^*\Sp_2))$, this is enough to show existence and uniqueness of a $L^0$-linear and continuous $\Phi_1$ for which \eqref{eq:reqphi1} holds and, from \eqref{eq:normPhi1}, that for such $\Phi_1$ we have 
\begin{equation}
\label{eq:leqPhi1}
|\Phi_1(W)|\leq |W|\qquad \mea_1\times\mea_2-a.e..
\end{equation}
Thus to conclude it is  sufficient to show that equality holds and, by the very same arguments just given, to this aim it is sufficient to show that
\begin{equation}
\label{eq:suffPhi1}
|\d g|\circ\pi_1= |\d(g\circ\pi_1)|\quad\mea_1\times\mea_2-a.e.\qquad\forall g\in\Sclass_\loc(\Sp_1).
\end{equation}
It is now convenient to consider the map
\[
\begin{split}
\mathrm{T} \colon \mathit{C}([0, 1], \Sp_1) \times \Sp_2 \quad&\rightarrow \quad\mathit{C}([0, 1], \Sp_1 \times \Sp_2)\\
(t\to \gamma_t, x_2)\quad & \mapsto\quad t\to (\gamma_t, x_2).
\end{split}
\]
Notice that 
\begin{equation}
\label{eq:samespeed}
\text{for any $x_2\in\X_2$ the speed of $T(\gamma,x_2)$ is equal to the speed of $\gamma$ for a.e.\ $t$}
\end{equation}
and  fix $\mu \in \proba(\Sp_2)$ such that $\mu \le \tilde{C} \mea_2$ for some $\tilde C>0$. Then for an arbitrary test plan $\bpi$ on $\Sp_1$   define
\[ 
\tilde{\bpi} := \mathrm{T}_{\ast}(\bpi \times \mu) \in \proba(\mathit{C}([0, 1], \Sp_1 \times \Sp_2))
\]
and observe that  $\tilde{\bpi}$ is a test plan on $\Sp_1\times\Sp_2$. Hence for any $g\in\Sclass_\loc(\Sp_1)$ we have
\[
\begin{split}
\displaystyle \int \abs*{g \circ e_1 - g \circ e_0} \ \dint \bpi &=  \int \abs*{g \circ \pi_1 \circ e_1 - g \circ \pi_1 \circ e_0} \ \dint \tilde{\bpi} \\
& \le \int_0^1 \int \abs*{\dint (g \circ \pi_1)}(\tilde{\gamma}_t) \abs{\dot{\tilde{\gamma}}_t} \ \dint \tilde{\bpi}(\tilde{\gamma}) \ \dint t \\
\text{by \eqref{eq:samespeed}}\qquad\qquad&= \displaystyle \int_0^1 \int \left( \int \abs*{\dint (g \circ \pi_1)}(\mathrm{T}(\gamma, x_2)_t) \ \dint \mu(x_2)  \right) \abs{\dot{\gamma}_t} \ \dint \bpi(\gamma) \ \dint t,
\end{split}
\]
so that  the arbitrariness of $\bpi$ gives that  the function $\int \abs*{\dint (g \circ \pi_1)}(\cdot,x_2) \dint \mu(x_2)$ is a weak upper gradient of $g$. Therefore for $\mea_1$-a.e.\ $x_1$ we have
\[ 
\abs{\dint g}(x_1) \le \int \abs*{\dint (g \circ \pi_1)}(x_1,x_2) \ \dint \mu(x_2)\stackrel{\eqref{diffbound}}\leq \int |\dint g| \circ \pi_1(x_1,x_2) \ \dint \mu(x_2)=\abs{\dint g}(x_1).
\]
Hence the inequalities are  equalities and the arbitrariness of $\mu$ gives \eqref{eq:suffPhi1}.
\end{proof}

It seems hard to obtain any further relation between the cotangent modules in full generality, for this reason from now on we shall make two structural assumptions. In the following, for any function $f (x_1, x_2)$ on the product space $\Sp_1 \times \Sp_2$, we define $f^{x_1}(\cdot) := f(x_1, \cdot)$ and, similarly, $f_{x_2}(\cdot) := f(\cdot, x_2)$. 
\begin{definition}[Tensorization of the Cheeger energy]\label{def:tensch}
We say that  two metric measure spaces $\mmso$ and $\mmst$ have the property of tensorization of the Cheeger energy provided for any $f \in \Lint^2(\Sp_1 \times \Sp_2)$ the following holds:  $ f\in \W(\Sp_1\times \Sp_2)$ if and only if
\begin{itemize}
\item[-]  $f^{x_1}\in\W (\Sp_2)$ for $\mea_1$-a.e. $x_1 \in \Sp_1$ and $\iint \abs{\dint f^{x_1}}^2\,\d\mea_2\,\d\mea_1(x_1)<\infty$ 
\item[-] $f_{x_2} \in \W (\Sp_1)$ for $\mea_2$-a.e. $x_2 \in \Sp_2$ with $\iint\abs{\dint f_{x_2}}^2\,\d\mea_1\,\d\mea_2(x_2)<\infty$
\end{itemize}
and in this case it holds
\begin{equation} \label{Chpr}
\abs{\dint f}^2 (x_1, x_2) = \abs{\dint f^{x_1}}^2(x_2) + \abs{\dint f_{x_2}}^2(x_1) \qquad\mea_1 \times \mea_2-a.e.\ (x_1,x_2).
\end{equation}
\end{definition}

Notice that for $ g \in \Lint^{\infty} \cap \W(\Sp_1)$ and $h \in \Lint^{\infty} \cap \W (\Sp_2)$ both with bounded support, the function $g\circ\pi_1\,h\circ\pi_2$ has bounded support and is in $\Lint^{\infty} \cap \W(\Sp_1\times\Sp_2)$, its differential being given by
\begin{equation}
\label{eq:pera}
\d (g\circ\pi_1\,h\circ\pi_2)=g\circ\pi_1\,\d(h\circ\pi_2)+h\circ\pi_2\,\d(g\circ\pi_1).
\end{equation}
\begin{definition}[Density of the product algebra]\label{def:densalg}
We say that  two metric measure spaces $\mmso$ and $\mmst$ have the property of density of product algebra  if the set
\begin{equation}
\label{eq:prodalg}
\mathcal{A} := \Big\{  \displaystyle \sum_{j=1}^n g_j\circ\pi_1\, h_j\circ\pi_2\ :\ n\in\N,\  \begin{array}{l}
g_j \in \Lint^{\infty} \cap \W(\Sp_1)\text{ has bounded support}\\
 h_j \in \Lint^{\infty} \cap \W (\Sp_2) \text{ has bounded support}
 \end{array}
\forall j=1,\ldots,n\Big\}
\end{equation}
is dense in $\W(\Sp_1 \times \Sp_2)$ in the strong topology of $\W(\Sp_1 \times \Sp_2)$.
\end{definition}

From now on we will always assume the following:
\begin{assumption}\label{ass} $\mmso$ and $\mmst$ are two metric measure spaces for which both the  tensorization of Cheeger energy and the density of the product algebra hold.
\end{assumption}
It is worth to underline that no couple of spaces $\Sp_1,\Sp_2$ are known for which Assumption \ref{ass} does not hold. On the other hand, it is unclear if that holds in full generality. The first results about the tensorization of Cheeger energies being given in \cite{AmbrosioGigliSavare11-2} for the cases of two $\RCD$ spaces with finite mass, for our purposes the following result covers the cases of interest:
\begin{proposition}\label{prop:prodrcd}
Let $(\Sp_1,\sfd_1,\mm_1)$ be a $\RCD(K,\infty)$ spaces and let $\Sp_2$ be the Euclidean space $\R^N$ equipped with the Euclidean distance and the Lebesgue measure.

Then both the tensorization of the Cheeger energy and the density of the product algebra hold.
\end{proposition}
\begin{proof}
In \cite{GH15} it has been proved that for arbitrary $\mmso$ and for $\Sp_2=\R$ the tensorization of the Cheeger energy holds and the algebra $\mathcal A$ is dense in energy, i.e.: for any $f\in \W(\Sp_1\times\R)$ there is $(f_n)\subset\mathcal A$ such that $f_n\to f$ and $|\d f_n|\to|\d f|$ in $L^2(\Sp_1\times\R)$. 

If $\mmso$ is infinitesimally Hilbertian (which is the case for $\RCD$ spaces), then the tensorization of the Cheeger energy ensures that $\W(\Sp_1\times\R)$ is a Hilbert space, so that the uniform convexity of the norm grants that convergence in energy implies strong $\W$-convergence.

Thus the thesis is true for $\Sp_2=\R$. The general case follows by a simple induction argument.
\end{proof}

With this said, we shall now continue the investigation of the relation between cotangent modules and products of spaces. We start with the following result; notice that  the density of the product algebra is used to show that for $f\in  \Sclass_\loc(\Sp_1 \times \Sp_2)$ the map  $x_2 \mapsto \dint f_{x_2} \in \Lint^0(T^*\Sp_1)$ is essentially separably valued.
\begin{lemma}\label{le:perprod}
Let $\mmso$ and $\mmst$ be two metric measure spaces satisfying Assumption \ref{ass}. 

Then for every $f \in \Sclass_\loc(\Sp_1 \times \Sp_2)$ we have that $f_{x_2} \in\Sclass_\loc(\Sp_1)$ for $\mea_2$-a.e.\ $x_2$ and the map $x_2\mapsto  \dint f_{x_2}$ belongs to $\Lint^0(\Sp_2, \Lint^0(\mathit{T}^{\ast}\Sp_1))$. Moreover, for $(f_n)\subset\Sclass_\loc(\Sp_1\times\Sp_2)$ we have
\begin{equation}
\label{eq:convd}
\dint f_n\to \dint f\text{ in }\Lint^0(T^*(\Sp_1\times\Sp_2))\qquad\Rightarrow\qquad \dint (f_n)_{\cdot}\to \dint f_{\cdot}\text{ in }\Lint^0(\Sp_2,\Lint^0(T^*\Sp_1)).
\end{equation}
 Similarly for the roles of  $\Sp_1$ and $\Sp_2$ inverted. Finally, the identity \eqref{Chpr} holds for any $f\in\Sclass_\loc(\Sp_1\times\Sp_2)$.
 \end{lemma}
\begin{proof} Let $f  = g\circ\pi_1 \, h\circ\pi_2$ for some $g \in\Lint^\infty\cap \W(\Sp_1)$ and $h \in\Lint^\infty\cap \W(\X_2)$ with bounded supports and notice that $\dint f_{x_2} = h(x_2) \, \dint g$ for every $x_2\in\Sp_2$. Hence $x_2\mapsto \dint f_{x_2} \in\Lint^2(\Sp_2, \Lint^2(\mathit{T}^{\ast}\mathrm{X}_1))$.

By linearity, the same holds for a generic $f\in\mathcal A$. Now notice that for an arbitrary $f\in \W(\Sp_1 \times \Sp_2)$, the identity \eqref{Chpr} yields
\begin{equation}
\label{eq:perl0}
\abs{\dint f_{x_2}}^2(x_1) \leq  \abs{\d f}^2(x_1,x_2) \qquad\mea_1 \times \mea_2-a.e.\ (x_1, x_2),
\end{equation}
and thus
\begin{equation}
\label{eq:perl0l0}
\||\dint f_{\cdot}-\dint \tilde f_{\cdot}|\|_{\Lint^2(\Sp_2, \Lint^2(\Sp_1))}\leq \|f-\tilde f\|_{\W(\Sp_1 \times \Sp_2)}.
\end{equation}
Hence for $f\in \W(\Sp_1\times\Sp_2)$ arbitrary, using the density of the product algebra we can find $(f_n)\subset \mathcal A$ $\W$-converging to $f$, so that from \eqref{eq:perl0l0} we see that $\d f_\cdot\in \Lint^2(\Sp_2, \Lint^2(T^*\Sp_1))$.

For general $f\in\Sclass_\loc(\Sp_1\times\Sp_2)$, find a sequence $(f_n)\subset \W(\Sp_1\times\Sp_2)$ as in \eqref{eq:apprsloc} and use the locality of the differential to get that \eqref{eq:perl0} holds even for  $f\in\Sclass_\loc(\Sp_1\times\Sp_2)$. Thus, since clearly $\d f_n\to \d f$ in $\Lint^0(T^*(\Sp_1\times\Sp_2))$, from \eqref{eq:perl0} we also get that $|\dint (f_n)_{\cdot}-\dint f_{\cdot}|\to 0$ in $\Lint^0(\Sp_2, \Lint^0(\Sp_1))$: this proves both that $\dint f_{\cdot}\in \Lint^0(\Sp_2,\Lint^0(T^*\Sp_1))$ and that  $\dint (f_n)_{\cdot}\to \dint f_{\cdot}$ in $\Lint^0(\Sp_2,\Lint^0(T^*\Sp_1))$. 

Since this latter convergence does not depend on the particular choice of the sequence $(f_n)\subset\Sclass_\loc(\Sp_1\times\Sp_2)$ such that $\dint f_n\to \dint f$ in $\Lint^0(T^*(\Sp_1\times\Sp_2))$, we proved also \eqref{eq:convd}.

The last claim follows along the same approximation argument using the continuity property \eqref{eq:convd} (and the analogous one with $\Sp_1$ and $\Sp_2$ inverted).
\end{proof}

Let  $\mathscr M_1,\mathscr M_2$ be two $L^0$-normed modules on a space $\Sp$. Then on the product $\mathscr M_1\times\mathscr M_2$ we shall consider the structure of $L^0$-normed module given by: the product topology, the multiplication by $L^0$-functions  given by $f(v_1,v_2):=(fv_1,fv_2)$ and the pointwise norm defined as
\[
|(v_1,v_2)|^2:=|v_1|^2+|v_2|^2.
\]
It is readily verified that these actually endow $\mathscr M_1\times\mathscr M_2$ with the structure of $L^0$-normed module. 

In particular, $L^0(\Sp_2,L^0(T^*\Sp_1))\times L^0(\Sp_1,L^0(T^*\Sp_2))$ is a $L^0(\Sp_1\times\Sp_2)$-normed module and we can define $\Phi_1\oplus\Phi_2$ as
\[
\begin{array}{cccc}
\Phi_1 \oplus \Phi_2 \colon& \Lint^0 (\Sp_2, \Lint^0(\mathit{T}^{\ast}\Sp_1)) \times \Lint^0(\Sp_1, \Lint^0(\mathit{T}^{\ast}\Sp_2)) &\rightarrow  &\Lint^0 (\mathit{T}^{\ast}(\Sp_1 \times \Sp_2))\\
&(\omega, \sigma) &\mapsto& \Phi_1(\omega) + \Phi_2(\sigma)
\end{array}
\]
We then have the following result:
\begin{theorem}\label{dectang}
Let $\mmso$ and $\mmst$ be two metric measure spaces such that Assumption \ref{ass} holds. Then $\Phi_1\oplus\Phi_2$  is an isomorphism of modules, i.e. it is $L^0(\Sp_1\times\Sp_2)$-linear, continuous, 
 surjective and  for every $\omega_\cdot\in \Lint^0 (\Sp_2, \Lint^0(\mathit{T}^{\ast}\Sp_1))$ and $\sigma_\cdot\in \Lint^0 (\Sp_1, \Lint^0(\mathit{T}^{\ast}\Sp_2))$ satisfies
\begin{equation}
\label{eq:phi12}
|\Phi_1(\omega_\cdot) + \Phi_2(\sigma_\cdot)|^2=|\omega_\cdot|^2+|\sigma_\cdot|^2\quad\mea_1\times\mea_2-a.e..
\end{equation}
Moreover, for every   $f \in \Sclass_\loc(\Sp_1 \times \Sp_2)$ it holds:
\begin{equation}
\label{PhiinW}
\dint f = \Phi_1 ( \dint f_{\cdot}) + \Phi_2 ( \dint f^{\cdot} ).
\end{equation}
\end{theorem}
\begin{proof}  From Proposition \ref{prop:defPhi}  it is clear that $\Phi_1\oplus\Phi_2$ is $L^0(\Sp_1\times\Sp_2)$-linear and continuous.  Taking into account that $ \Lint^0 (\Sp_2, \Lint^0(\mathit{T}^{\ast}\Sp_1))$ is generated by elements of the kind $\widehat{\dint g}$ for $g\in \Sclass(\Sp_1)$, where $\widehat{\dint g}\in  \Lint^0 (\Sp_2, \Lint^0(\mathit{T}^{\ast}\Sp_1))$ is the function identically equal to $\dint g$, and  similarly for  $ \Lint^0 (\Sp_1, \Lint^0(\mathit{T}^{\ast}\Sp_2))$, to prove \eqref{eq:phi12} it is sufficient to show that
\begin{equation}
\label{eq:phi122}
|\Phi_1(\widehat{\dint g}) + \Phi_2(\widehat{\dint h})|^2=|\dint g|^2\circ\pi_1+|\dint h|^2\circ\pi_2\quad\mea_1\times\mea_2-a.e.
\end{equation}
for any $g\in\Sclass(\Sp_1)$, $h\in\Sclass(\Sp_2)$. Fix such $g,h$ and put $f:=g\circ\pi_1+h\circ\pi_2\in\Sclass_\loc(\Sp_1\times\Sp_2)$. Notice that trivially $\dint f_{x_2}=\dint g$ and $\dint f^{x_1}=\dint h$ for any $x_1\in\Sp_1$ and $x_2\in\Sp_2$, hence from the tensorization of Cheeger energy (recall the last claim of Lemma \ref{le:perprod} above) we have
\[
|\Phi_1(\widehat{\dint g}) + \Phi_2(\widehat{\dint h})|^2=|\dint (g\circ\pi_1)+\dint (h\circ\pi_2) |^2=|\dint f|^2\stackrel{\eqref{Chpr}}=
|\dint f_{\cdot}|^2+|\dint f^{\cdot}|^2=|\dint g|^2\circ\pi_1+|\dint h|^2\circ\pi_2
\]
which is \eqref{eq:phi122}. Thus $\Phi_1\oplus\Phi_2$ preserves the pointwise norm.

Now we prove  \eqref{PhiinW}. Let $g\in \Lint^\infty\cap \W(\Sp_1)$ and $h\in \Lint^\infty\cap \W(\Sp_2)$ be both with bounded support and consider $f:=g\circ\pi_1\,h\circ\pi_2$. Then $f\in \W(\Sp_1\times\Sp_2)$ and  the very definition of $\Phi_1,\Phi_2$ grant that
\[
\d f=h\circ\pi_2\,\dint(g\circ\pi_1)+g\circ\pi_1\,\dint(h\circ\pi_2)=\Phi_1(h\,\dint g)+\Phi_2(g\,\dint h)=\Phi_1 \big( \dint f_{\cdot}\big) + \Phi_2 \big( \dint f^{\cdot} \big),
\]
so that in this case \eqref{PhiinW} is proved. By linearity, we get that \eqref{PhiinW} holds  for general  $f\in\mathcal A$. Then using first the density of $\mathcal A$ in $\W(\Sp_1\times\Sp_2)$ and then property \eqref{eq:apprsloc}, taking into account the convergence property \eqref{eq:convd} we conclude that \eqref{PhiinW} holds for general $f \in \Sclass_\loc(\Sp_1 \times \Sp_2)$, as claimed.

It remains to prove that $\Phi_1\oplus\Phi_2$ is surjective. By  \eqref{PhiinW} we know that its image contains the space of differential of functions in $\Sclass_\loc(\Sp_1\times\Sp_2)$, and thus $L^0$-linear combinations of them. Since it preserves the pointwise norm, its image must be closed and since $\Lint^0 (\mathit{T}^{\ast}(\Sp_1 \times \Sp_2))$ is generated by differentials of functions in $\Sclass_\loc(\Sp_1\times\Sp_2)$, this is sufficient to conclude.
\end{proof}

\subsubsection{Other differential operators in product spaces}\label{se:other}

In the previous section we have seen how the differential behaves under products of spaces. We shall now investigate other differentiation operators under the assumption that $\Sp_1,\Sp_2$ are infinitesimally Hilbertian.

We  start with the following simple orthogonality statement:
\begin{proposition} Let $\mmso$ and $\mmst$ be  infinitesimally Hilbertian spaces such that Assumption \ref{ass} holds. Then $\Sp_1\times\Sp_2$ is also infinitesimally Hilbertian and for every $\omega^1_\cdot\in \Lint^0 ( \Sp_2, \Lint^0(\mathit{T}^* \Sp_1) )$ and $\omega^2_\cdot\in \Lint^0 ( \Sp_1, \Lint^0(\mathit{T}^* \Sp_2) )$ we have
\begin{equation}
\label{eq:orto}
\la \Phi_1(\omega^1_\cdot),\Phi_2(\omega^2_\cdot)\ra=0\quad\mea_1\times\mea_2-a.e..
\end{equation}
\end{proposition}
\begin{proof} The fact that $\W(\Sp_1\times\Sp_2)$ is Hilbert is a direct consequence of the tensorization of the Cheeger energy and the assumption that both $\W(\Sp_1)$ and $\W(\Sp_2)$ are Hilbert. For \eqref{eq:orto} notice that 
\[
|\omega_\cdot^1|^2+|\omega^2_\cdot|^2\stackrel{\eqref{eq:phi12}}=| \Phi_1(\omega^1_\cdot)+\Phi_2(\omega^2_\cdot)|^2=| \Phi_1(\omega^1_\cdot)|^2+|\Phi_2(\omega^2_\cdot)|^2+2\la \Phi_1(\omega^1_\cdot),\Phi_2(\omega^2_\cdot)\ra,
\]
so that the conclusion follows recalling that $\Phi_1,\Phi_2$ preserve the pointwise norms.
\end{proof}
By means of the musical isomorphisms (recall \eqref{eq:music}) the map $\Phi_1$ induces a map, still denoted  $\Phi_1$, from $\Lint^0 ( \Sp_2, \Lint^0(\mathit{T} \Sp_1) )$ to $\Lint^0 (\mathit{T} (\Sp_1\times \Sp_2))$ via:
\[
\Phi_1(X_\cdot):=\Phi_1(X_\cdot^\flat)^\sharp.
\]
Similarly for $\Phi_2$. It is clear that these newly defined $\Phi_1,\Phi_2$ have all the properties we previously proved for the same operators viewed as acting on forms. We also notice that for any $\omega_{\cdot}\in \Lint^0(\Sp_2,\Lint^0(T^*\Sp_1))$ and $X_{\cdot}\in \Lint^0(\Sp_2,\Lint^0(T\Sp_1))$ we have
\begin{equation}
\label{eq:2phi1}
\Phi_1(\omega_\cdot)(\Phi_1(X_{\cdot}))(x_1,x_2)=\omega_{x_2}(X_{x_2})(x_1)\quad\mea_1\times\mea_2-a.e.\ (x_1,x_2).
\end{equation}
Indeed, for $\omega_\cdot\equiv \dint g$ and $X_\cdot\equiv\nabla\tilde g$ for $g,\tilde g\in\Sclass_\loc(\Sp_1)$ this is a direct consequence of the definition of $\Phi_1$ and the fact that $\Phi_1$ preserves the pointwise norm (and hence the pointwise scalar product), then the general case follows by $L^0(\Sp_1\times\Sp_2)$-bilinearity and continuity of both sides.

\begin{proposition}\label{divvf}
Let $\mmso$ and $\mmst$ be  infinitesimally Hilbertian spaces such that Assumption \ref{ass} holds. Then $ X \in D( \Div_{\loc}, \Sp_1)$ if and only if $\Phi_1(\hat X) \in D (\Div_{\loc}, {\Sp_1 \times \Sp_2})$, where $\hat X\in \Lint^0(\Sp_2,\Lint^0(T\Sp_1))$ is the function identically equal to $X$,  and in this case 
\[
\Div(\Phi_1(\hat X)) = \Div (X) \circ \pi_1.
\]
\end{proposition}
\begin{proof} From the very definition of divergence it is readily verified that the thesis is equivalent to
\[
\int \dint f(\Phi_1(\hat X))\,\d(\mea_1\times\mea_2)=\iint \dint f_{\cdot}(X)\,\d\mea_1\,\d\mea_2
\]
for every $f\in\W(\Sp_1\times\Sp_2)$ with bounded support.

For such $f$ we have
\[
\begin{split}
\int \dint f(\Phi_1(\hat X))\,\d(\mea_1\times\mea_2)&\stackrel{\eqref{PhiinW}}=\int \big(\Phi_1 \big( \dint f_{\cdot}\big) + \Phi_2 \big( \dint f^{\cdot} \big)\big)(\Phi_1(\hat X))\,\d(\mea_1\times\mea_2)\\
&\stackrel{\eqref{eq:orto}}=\int \big(\Phi_1 \big( \dint f_{\cdot}\big) \big)(\Phi_1(\hat X))\,\d(\mea_1\times\mea_2)\\
&\stackrel{\eqref{eq:2phi1}}=\iint  \dint f_{\cdot}(X)\,\d \mea_1\,\d \mea_2,
\end{split}
\]
hence the conclusion.
\end{proof}
A related property is the following:

\begin{proposition}\label{divprod}
Let $\mmso$ and $\mmst$ be  infinitesimally Hilbertian spaces such that Assumption \ref{ass} holds. Let $X = \Phi_1(X^1_{\cdot}) + \Phi_2(X^2_{\cdot}) \in \Lint^2(\mathit{T}(\Sp_1 \times \Sp_2))$ be such that:
\begin{itemize}
\item[-] $X^1_{x_2} \in D(\div, \Sp_1)$ for $\mea_2$-a.e. $x_2 \in \Sp_2$ with $\displaystyle \int \ \big|{\div (X^1_{\cdot})}\big|^2  \ \dint (\mea_1 \times \mea_2) < \infty$
\item[-] $X^2_{x_1} \in D(\div, \Sp_2)$ for $\mea_1$-a.e. $x_1 \in \Sp_1$ with $\displaystyle \int \ \big|{\div (X^2_{\cdot})}\big|^2  \ \dint (\mea_1 \times \mea_2) < \infty$.
\end{itemize}
Then $X \in D (\Div)$ and 
\begin{equation}
\label{eq:divprod}
\Div(X) (x_1, x_2) = \div(X^1_{x_2}) (x_1) + \Div(X^2_{x_1}) (x_2)\qquad\mea_1\times\mea_2-a.e.\ (x_1,x_2).
\end{equation}
\end{proposition}
\begin{proof}
For any $f\in\W(\Sp_1\times\Sp_2)$ with bounded support we have
\[
\begin{split}
\int\dint f(X)\,\dint(\mea_1 \times \mea_2) &\stackrel{\eqref{PhiinW}}=\int  \big(\Phi_1 \big( \dint f_{\cdot}\big) + \Phi_2 \big( \dint f^{\cdot} \big)\big)(\Phi_1( X_\cdot^1)+\Phi_2(  X^2_\cdot))\,\d(\mea_1\times\mea_2)\\
&\stackrel{\eqref{eq:orto}}=\int  \Phi_1 \big( \dint f_{\cdot}\big)\Phi_1( X_\cdot^1) + \Phi_2 \big( \dint f^{\cdot} \big)\Phi_2(  X^2_\cdot)\,\d(\mea_1\times\mea_2)\\
&\stackrel{\eqref{eq:2phi1}}=  \int\Big( \int \d f_{\cdot}(X^1_\cdot)\,\d \mea_1\Big)\d \mea_2+\int\Big(\int\dint f^{\cdot}(X^2_\cdot)\,\d \mea_2\Big)\d \mea_2,
\end{split}
\]
which is the thesis.
\end{proof}
These last two statements produce analogous ones for the Laplacian:
\begin{corollary}\label{cor:lapprod}
Let $\mmso$ and $\mmst$ be  infinitesimally Hilbertian spaces such that Assumption \ref{ass} holds.  Then:
\begin{itemize}
\item[i)] $f \in D(\Delta_\loc, \Sp_1)$ if and only if $f \circ \pi_1 \in D(\Delta_\loc, \Sp_1 \times \Sp_2)$ and in this case 
\[
\Delta(f \circ \pi_1) = (\Delta f) \circ \pi_1.
\]
\item[ii)] Let $f \in \W(\Sp_1 \times \Sp_2)$ be such that
\begin{itemize}
\item for $\mea_1$-a.e. $x_1 \in \Sp_1$, $f^{x_1} \in D(\Delta, \Sp_2)$ with $\int \norma{\Delta f^{x_1}}^2_{\Lint^2(\Sp_2)} \ \dint \mea_1 < \infty$,\\
\item for $\mea_2$-a.e. $x_2 \in \Sp_2$, $f_{x_2} \in D(\Delta, \Sp_1)$ with $\int \norma{\Delta f_{x_2}}^2_{\Lint^2(\Sp_1)} \ \dint \mea_2 < \infty$.
\end{itemize}
Then $f \in D(\Delta, \Sp_1 \times \Sp_2)$ and
\begin{equation}\label{prodlap}
\Delta f (x_1,x_2) = \Delta f_{x_2} (x_1) + \Delta  f^{x_1}(x_2)\qquad\mea_1\times\mea_2-a.e.\ (x_1,x_2).
\end{equation}
\end{itemize} 
\end{corollary}
\begin{proof}
For the first claim simply notice that, directly from the definition, we have $f \in D(\Delta_\loc, \Sp_1)$ if and only if $\nabla f \in D(\Div_{\loc}, \Sp)$ and in this case $\div(\nabla f)=\Delta f$. Similarly for $f\circ\pi_1$. Then observe that \eqref{eq:reqphi1} grants that $\nabla(f\circ\pi_1)=\Phi_1(\widehat{\nabla f})$ and apply Proposition \ref{divvf} above to conclude.

The second claim follows by analogous considerations using Proposition \ref{divprod}  and the identity $\nabla f= \Phi_1(\nabla f_{\cdot}) + \Phi_2(\nabla f^{\cdot})$ (recall \eqref{PhiinW}).
\end{proof}

\subsection{Flow of harmonic vector fields on $\RCD(0,\infty)$ spaces} \label{isox}
In this section we work on a fixed $\RCD(0,\infty)$ space $\mms$ and study the Regular Lagrangian Flow of a fixed non-zero vector field $X\in\Lint^2(T\Sp)$ which is {\bf harmonic}, i.e.\ $X^\flat\in D(\Delta_H)$ with $\Delta_HX^\flat=0$. Recalling  \eqref{eq:hlapdiv} we have that  $\div X=0$, while \eqref{eq:energiah} grants  that $X$ is parallel, i.e.\ $X\in H^{1,2}_C(T\Sp)$ with $\nabla X=0$. This latter property also implies that  $|X|$ is constant (see \cite{Gigli14} for the details about this last claim).

We can thus apply Theorem \ref{thm:AT} to deduce that there exists and is unique the Regular Lagrangian Flow $(\Fl^{(X)}_t)$ of $X$. Aim of this section is to prove that:
\begin{itemize}
\item[i)] the $\Fl^X_t$'s are measure preserving isometries
\item[ii)] if $Y$ is another harmonic vector field, then $\Fl^X_t\circ \Fl^Y_s=\Fl^{tX+sY}_1$ for any $t,s$. 
\end{itemize}
Notice that by analogy with the smooth case,  one would expect to need only the conditions $\div X=0$, $\nabla X=0$ and that $\Sp$ is a $\RCD(K,\infty)$ space to get the above. Yet, it is unclear to us whether these are really sufficient, (part of) the problem being in the approximation procedure used in Proposition \ref{prop:euler} which requires our stronger assumptions.

\bigskip

In what comes next  we shall occasionally use the following simple fact: for $T,S\colon\Sp\to\Sp$ Borel we have
\begin{equation}
\label{eq:mappeuguale}
T_*\mu=S_*\mu\quad\text{$\forall\mu\in\prob\Sp$ with bounded support and density}\qquad\Rightarrow\qquad T=S\quad\mm-a.e..
\end{equation}
Indeed, if $T\neq S$ on a set of positive measure,  for some $r>0$ we would have $\sfd(T(x),S(x))>2r$ for a set of $x$'s of positive measure and thus using the separability of $\Sp$ we would be able to find $\bar x$ such that $T_*\mm(B_r(\bar x))>0$. Thus  $\mm(T^{-1}(B_r(\bar x)))>0$ and letting $A\subset T^{-1}(B_r(\bar x))$ be any bounded Borel subset of positive $\mm$-measure, for $\mu:=\mm(A)^{-1}\mm\restr A$ we would have that $T_*\mu$ and $S_*\mu$ are concentrated on disjoint sets, and thus in particular $T_*\mu\neq S_*\mu$.

With this said, we prove the following result, which shows that the flows of  $X$ and $-X$ are one the inverse of the other:
\begin{lemma}\label{invFl}
Let $(\X,\sfd,\mm)$ be a $\RCD(0,\infty)$ space and $X$ a harmonic vector field.  Then for every $t \geq 0$ the following identities hold $\mm$-a.e.:
\[
\Fl_{t}^{(-X)}\circ \Fl_{t}^{(X)}=\Id
 \qquad \text{and} \qquad 
\Fl_t^{(X)}\circ \Fl_{ t}^{(-X)}=\Id.
\]
\end{lemma}
\begin{proof} We shall prove the first identity for $t=1$, as then the rest  follows by similar arguments. Let  $\mu\in\proba(\Sp)$ be with bounded support and density, and consider the curves $[0,1]\ni t\mapsto \mu_t,\tilde\mu_t\in \proba(\Sp)$ defined as
\[
\begin{split}
\mu_t:=(\Fl^{(X)}_{1-t})_{\ast} \mu \qquad\text{ and }\qquad  \tilde\mu_t:=(\Fl^{(- X)}_{t})_{\ast} (\Fl^{(X)}_{1})_{\ast}\mu,
\end{split}
\] 
notice that $\mu_0=\tilde\mu_0$ and that  they both solve the continuity equation  \eqref{eq:ce} for $X_t = - X$ in the sense of Definition \ref{def:solce}.  By  Theorem \ref{thm:unice} we conclude that $\mu_1=\tilde\mu_1$, i.e.
\[
\mu=(\Fl_{1}^{(-X)}\circ \Fl_{1}^{(X)})_*\mu.
\]
The conclusion follows by the arbitrariness of $\mu$ and \eqref{eq:mappeuguale}.
\end{proof}
From this proposition and the semigroup property \eqref{eq:semigrpr} of Regular Lagrangian Flows, it follows that defining $\Fl^{(X)}_{-t}:=\Fl^{(-X)}_t$ for $t\geq 0$ we have
\begin{equation}\label{grpr}
\Fl_t^{(X)}\circ \Fl_{s}^{(X)} = \Fl^{(X)}_{t+s}  \qquad\mm \text{-a.e.}\qquad \forall \  t, s \in\R.
\end{equation}

\begin{proposition}[Preservation of the measure] \label{MP} Let $(\X,\sfd,\mm)$ be a $\RCD(0,\infty)$ space and $X$ a harmonic vector field.  
Then for every $t\in \real$ we have 
\begin{equation}
\label{eq:mespre}
(\Fl^{(X)}_t)_*\mm=\mm.
\end{equation}
\end{proposition}
\begin{proof} Simply notice that from $\div(X)=\div(-X)=0$ and \eqref{eq:quantbound}, for any $t\geq 0$ we have
\[
\mm=(\Fl^{(X)}_t\circ\Fl^{(X)}_{-t})_*\mm=(\Fl^{(X)}_t)_*(\Fl^{(-X)}_{t})_*\mm\leq (\Fl^{(X)}_t)_* \mm\leq\mm 
\]
forcing the inequalities to be equalities.
\end{proof}
Recall that the heat flow $(\h_t)$ on $\X$ is firstly defined as the $L^2$-gradient flow of the Dirichlet and then extended to a flow on $L^1+L^\infty$ by monotonicity and continuity. On the other hand, the `Hodge' heat flow $(\h_{H,t})$ is defined on $L^2(T^*\X)$ as the gradient flow of the functional
\[
L^2(T^*\X)\ni\omega\quad\mapsto\quad \left\{\begin{array}{ll}
\displaystyle{\frac12\int |\d\omega|^2+|\delta\omega|^2\,\d\mm}&\qquad\text{ if }\omega\in H^{1,2}_H(T^*\X),\\
+\infty&\qquad\text{ otherwise.}
\end{array}\right.
\]
For us it will be relevant to know the relation
\begin{equation}
\label{eq:relcal}
\h_{H,t}\d f=\d \h_tf\qquad\forall f\in W^{1,2}(\X)
\end{equation}
and the improved Bakry-\'Emery estimate:
\begin{equation}
\label{eq:beimp}
|\h_{H,t}\omega|^2\leq e^{-Kt} \h_t(|\omega|^2),\quad\mm-a.e.\qquad \ \forall t\geq 0,\ \omega\in L^2(T^*\Sp),
\end{equation}
valid on $\RCD(K,\infty)$ spaces. See \cite{Gigli14} for further details about these.

With this said, we can now prove the following lemma, which is key to show that $\Fl^{(X)}_t$ is an isometry.
\begin{proposition}[Euler's equation for $X$]\label{prop:euler} Let $(\X,\sfd,\mm)$ be a $\RCD(0,\infty)$ space and $X$ a harmonic vector field. Then for any $f \in \W(\Sp)$ it holds
\begin{equation} \label{Euler}
\h_{t}\left( \big\langle \nabla f, X \big\rangle \right) = \big\langle  \nabla \h_{t} f, X \big\rangle,  \quad \mae, \forall t \ge 0.
\end{equation}
Moreover, for every $f \in D(\Delta)$ with $\Delta f \in \W(\Sp)$, we have $\langle \nabla f, X \rangle \in D(\Delta)$ and
\begin{equation} \label{Nabla}
\Delta \big\langle \nabla f, X \big\rangle = \big\langle  \nabla\Delta f, X \big\rangle, \quad \mae.
\end{equation}
\end{proposition}
\begin{proof} We apply \eqref{eq:beimp} in our space to the form $X^\flat+\eps \d f$ to obtain
\begin{equation}
\label{eq:stap1}
\abs{\h_{H, t}(X^{\flat} + \varepsilon \, \dint f)}^2 \le \h_{t} (\abs{X + \varepsilon \nabla f }^2).  
\end{equation}
We have already observed that the  fact that $X^\flat$ is  harmonic  grants that $|X|$ is constant, say $|X|\equiv c$. The harmonicity also grants  that $\h_{H, t}(X^\flat) = X^\flat$ for every $t \ge 0$, hence we have $\abs{\h_{H, t}(X^\flat)}^2 \equiv c^2\equiv\h_t( \abs{X}^2)$ for any $t\geq 0$. Therefore,
\[  c^2 + 2 \varepsilon \langle X, \h_{H, t}(\dint f) \rangle + \varepsilon^2 \abs{\h_{H, t}(\dint f)}^2 \le c^2 + 2 \varepsilon \h_t \langle X, \nabla f \rangle + \varepsilon^2 \h_t(\abs{\dint f}^2) \]
and the arbitrariness of $\eps\in\R$ implies
\[  
\langle X, \h_{H, t} \dint f \rangle = 2 \h_{t} \langle X, \nabla f \rangle,
\]
which by  \eqref{eq:relcal} is \eqref{Euler}. Then \eqref{Nabla} comes by differentiating \eqref{Euler} at $t=0$.
\end{proof}

\begin{proposition}[Preservation of the Dirichlet energy]\label{DE}
 Let $(\X,\sfd,\mm)$ be a $\RCD(0,\infty)$ space and $X$ a harmonic vector field. Then for every $t\in\R $ we have
 \begin{equation}
\label{eq:sameener}
 \ener(f \circ \Fl^X_t) = \ener(f)\qquad\forall f\in W^{1,2}(\X).
\end{equation}
\end{proposition}
\begin{proof} Fix $f\in W^{1,2}(\X)$, put $f_t:=f\circ\Fl^X_t$ and notice that since $\ener(\h_\eps g)\to\ener (g)$ as $\eps\downarrow0$ for any $g\in\Lint^2(\Sp)$, it is sufficient to prove that for any $\eps>0$ we have
\[
\ener(\h_\eps f_t)=\ener(\h_\eps f)\qquad\forall t\in\R.
\]
Thus fix $\eps>0$ and notice that Proposition \ref{prop:rlfeq} grants that $t\mapsto f_t\in L^2(\Sp)$ is Lipschitz. This in conjunction with the fact that $\h_\eps: L^2(\Sp)\to \W(\Sp)$ is continuous  ensures that $ t \mapsto  \h_{\varepsilon}f_t \in \W(\Sp)$ is Lipschitz.

We now compute the derivative of  the Lipschitz map $  t \mapsto \ener(\h_\eps f_t)$ and start noticing that
\[ 
\displaystyle \int \abs{\nabla \h_{\varepsilon}f_{t+h}}^2 - \abs{\nabla \h_{\varepsilon}f_t}^2 \dint\mea = \int \abs{\nabla\h_{\varepsilon}(f_{t+h} - f_t)}^2 + 2 \big\langle\nabla \h_{\varepsilon}f_t, \nabla \h_{\varepsilon}(f_{t+h} - f_t)\big\rangle \dint\mea,
\]
so that the Lipschitz regularity of  $ t \mapsto  \h_{\varepsilon}f_t \in \W(\Sp)$ grants that for any $t\in\R$ it holds
\[
\lim_{h \rightarrow 0} \displaystyle \int{\dfrac{\abs{\nabla \h_{\varepsilon}f_{t+h}}^2 - \abs{\nabla \h_{\varepsilon}f_t}^2}{2h}} \dint\mea = \lim_{h \rightarrow 0} \int
\Big\langle \nabla \h_{\varepsilon}f_t, \nabla \dfrac{\h_{\varepsilon}f_{t+h} - \h_{\varepsilon}f_t}{h} \Big\rangle \dint\mea.
\]
Hence
\[
\begin{split}
\frac{\d}{\d t}\ener(\h_\eps f_t)&= -\lim_{h \rightarrow 0} \displaystyle \int \Delta \h_{\varepsilon}f_t \  \dfrac{\h_{\varepsilon}f_{t+h} - \h_{\varepsilon} f_t}{h} \dint\mea \notag\\
&= -\lim_{h \rightarrow 0} \displaystyle \int \Delta \h_{2\varepsilon}f_t \ \dfrac{f_{t} \circ \Fl^{(X)}_h - f_t}{h} \dint\mea \notag\\
\text{by \eqref{eq:mespre}}\qquad\qquad&= -\lim_{h \rightarrow 0} \displaystyle \int \dfrac{\big(\Delta \h_{2\varepsilon}f_t\big) \circ \Fl^{(X)}_{-h} - \Delta \h_{2\varepsilon}(f_t)}{h} f_t \dint\mea \notag\\
\text{by the last claim in Proposition \ref{prop:rlfeq}}\qquad\qquad&{=} -\displaystyle \int  \langle\nabla \Delta \h_{2\varepsilon} f_t,X\rangle \, f_t \dint\mea .
\end{split}
\]
To conclude it is therefore sufficient to prove that for any $g\in \Lint^2(\Sp)$ it holds
\begin{equation} \label{PrScZero} 
 \int  \langle\nabla \Delta \h_{2\varepsilon} g,X\rangle \, g \,\dint\mea =0
 \end{equation}
Hence fix $g\in \Lint^2(\Sp)$ and notice that
\begin{equation} \label{prima}
\begin{split}
 \int  \langle\nabla \Delta \h_{2\varepsilon} g,X\rangle \, g \,\dint\mea\stackrel{\eqref{Euler}}= \int \h_{\varepsilon} \langle\nabla \Delta \h_{\varepsilon}g,X\rangle  g \,\dint\mea =\int  \langle\nabla \Delta \h_{\varepsilon}g,X\rangle  \h_{\varepsilon} g\, \dint\mea 
\end{split}
\end{equation}
and, recalling that $\div X=0$, that
\[
\int  \langle\nabla \Delta \h_{\varepsilon}g,X\rangle \, \h_{\varepsilon} g \,\dint\mea =-\int  \Delta \h_{\varepsilon}g\,\langle X, \nabla \h_{\varepsilon} g \rangle\,\dint\mea \stackrel{\eqref{Nabla}}=-\int\h_{\varepsilon} g\,  \langle X, \nabla\Delta \h_{\varepsilon}g\rangle  \,\dint\mea.
\]
This proves \eqref{PrScZero} and the theorem.
\end{proof}

We therefore can conclude that:
\begin{theorem}\label{thm:isofl} Let $(\X,\sfd,\mm)$ be a $\RCD(0,\infty)$ space and $X$ a harmonic vector field.   Then for every $t\in\R$ the map $\Fl^{(X)}_t$ has a continuous representative and this representative is  a measure preserving isometry.
\end{theorem}
\begin{proof} Use  the preservation of measure  proved in Proposition \ref{MP} and the one of Dirichlet energy proved in Proposition \ref{DE} in conjunction  with Theorem \ref{thm:isorcd}.
\end{proof}
From now on we shall identify $\Fl^{(X)}_t$ with its continuous representative. It is readily verified from the construction that the group property \ref{grpr} holds everywhere. 

One of the consequences of the fact that $\Fl^{(X)}_t$ is an automorphism of $\mms$ is that
\begin{equation}
\label{eq:auto}
f\in \test\Sp \qquad\Rightarrow\qquad f\circ\Fl^{(X)}_t\in\test\Sp.
\end{equation}
This can be seen by noticing that since $\Fl^{(X)}_t$ is a measure preserving isometry, directly from the definition of Sobolev space we have
\[
f\in W^{1,2}(\X)\quad\Leftrightarrow\quad f\circ \Fl^{(X)}_t\in W^{1,2}(\X)\quad\text{and in this case }|\d f|\circ \Fl^{(X)}_t=|\d(f\circ \Fl^{(X)}_t)|.
\]
From this fact and the definition of Laplacian we then deduce that
\[
f\in D(\Delta)\quad\Leftrightarrow\quad f\circ \Fl^{(X)}_t\in  D(\Delta) \quad\text{and in this case }(\Delta f)\circ \Fl^{(X)}_t=\Delta(f\circ \Fl^{(X)}).
\]
A suitable iteration of these arguments then yields \eqref{eq:auto}. 

Recall also (see \cite{Gigli14}) that being $\Fl^X_t$ invertible and of bounded deformation, its differential $\d \Fl^X_t$ is a map from $\Lint^2(T\Sp)$ into itself (well)  defined by:
\begin{equation}\label{def:dFl}
\begin{split}
\d f(\d\, \Fl^X_s(Y))&=\d (f\circ\Fl^X_s) (Y)\circ \Fl^X_{-s}\qquad\forall f\in W^{1,2}(\X).
\end{split}
\end{equation}

\bigskip

We now want to prove that if $X,Y$ are both harmonic, their flows commute. The proof is based on   the following lemma:
\begin{lemma}\label{lem:dFl}  Let $(\X,\sfd,\mm)$ be a $\RCD(0,\infty)$ space and $X,Y$  harmonic vector fields. Then 
\[
\d\, \Fl^{(X)}_s(Y)=Y\qquad\forall s\in\R.
\]
\end{lemma}
\begin{proof} Since differential of test functions generate the whole cotangent module, the claim will follow if we show that for any $f\in \test\X$ the map $\R\ni s\mapsto \d f(\d \Fl^{(X)}_s(Y))$ is constant.

Taking account of the equality in (\ref{def:dFl}) and recalling \eqref{eq:auto}, in order to conclude it is sufficient to prove that for any $f\in\test\Sp$
\[
\frac{\d (f\circ\Fl^{(X)}_h)(Y)\circ\Fl^{(X)}_{-h}-\d f(Y)}h \, \text{ goes to  0 in the strong $\Lint^2(\Sp)$-topology as $h\to 0$}.
\]
To this aim, start observing that
\[
\begin{split}
\frac{\d (f\circ\Fl^{(X)}_h)(Y)\circ\Fl^{(X)}_{-h}-\d f(Y)}h&=\d\bigg(\frac{f\circ \Fl^{(X)}_h-f}h\bigg)(Y)\circ \Fl^{(X)}_{-h}+\frac{\d f(Y)\circ \Fl^{(X)}_{-h}-\d f(Y)}h.
\end{split}
\]
Since $f\in\test\X$ and $X\in H^{1,2}_C(T\X)$ we have $\d f(Y)\in W^{1,2}(\X)$ (see \cite{Gigli14} for details about this implication) and thus from the last claim in Proposition \eqref{prop:rlfeq} we have
\[
\lim_{s\to 0}\frac{\d f(Y)\circ \Fl^{(X)}_{-s}-\d f(Y)}s=-\d(\d f(Y))(X)\qquad\text{ in }L^2(\Sp),
\] 
hence  to conclude it is sufficient to show that
\begin{equation}
\label{eq:perchiudere}
\lim_{s\to 0}\d\bigg(\frac{f\circ \Fl^{(X)}_s-f}s\bigg)(Y)\circ \Fl^{(X)}_{-s}=\d(\d f(X))(Y).
\end{equation}
Let us start proving that
\begin{equation}
\label{eq:c1}
\frac{f\circ \Fl^{(X)}_s-f}s\quad\to\quad \d f(X)\quad\text{as }s\to0\qquad\text{ in }W^{1,2}(\X).
\end{equation}
Notice that \eqref{eq:contl2loc} grants convergence in  $L^2(\Sp)$; moreover the bound
\[
\begin{split}
\Big|\d\Big(\frac{f\circ \Fl^{(X)}_s-f}s\Big)\Big|^2&=\Big|\frac1s\int_0^s\d\big(\d f(X)\circ\Fl^{(X)}_r\big)\,\d r\Big|^2\\
&\leq \frac1s\int_0^s|\d\big(\d f(X)\circ\Fl^{(X)}_r\big)|^2\,\d r=\frac1s\int_0^s|\d(\d f(X))|^2\circ\Fl^{(X)}_r\,\d r
\end{split}
\] 
and the fact that $(\Fl^X_{r})_*\mm=\mm$ grant that  $\lims_{s\to0}\|\frac{f\circ \Fl^{(X)}_s-f}s\|_{W^{1,2}}\leq\|\d f(X)\|_{W^{1,2}}$ which is sufficient to get \eqref{eq:c1}. 

From \eqref{eq:c1} we deduce that 
\[
\d\bigg(\frac{f\circ \Fl^{(X)}_s-f}s\bigg)(Y)\quad\to\quad \d(\d f(X))(Y)\quad\text{as }s\to0\qquad\text{ in }\Lint^2(\Sp)
\]
hence \eqref{eq:perchiudere} follows from \eqref{eq:contl2loc}.
\end{proof}

\begin{theorem}\label{thm:sumflow}
Let $\mms$ be a $\RCD(0, \infty)$ space and $X, Y \in \Lint^2(\Tan)$ be two harmonic vector fields. Then for any $t,s \in \real$ it holds
\[  
\Fl^{X}_t \circ \Fl^Y_s = \Fl_1^{tX+sY}.
\]
\end{theorem}
\begin{proof}
For any $r \in \real$ consider the map $G_r := \Fl^{X}_{rt} \circ \Fl^Y_{rs}$. Now take $f \in \W(\Sp)$ and observe that $f \circ \Fl^X_{rt} \in \W(\Sp)$, as a consequence of Theorem \ref{thm:isofl}, and that from Proposition \ref{prop:rlfeq} it easily follows that $r\mapsto f\circ G_r\in L^2(\X)$ is Lipschitz. By direct computation we have:
\begin{equation}
\label{eq:derr}
\frac\d{\d r} (f \circ G_r) = \frac{\d}{\d r}(f \circ \Fl^X_{rt} \circ \Fl^Y_{rs}) = s\,\d (f \circ \Fl_{rt}^X)(Y) \circ \Fl^Y_{rs} + t\,\d f(X) \circ \Fl^X_{rt} \circ \Fl^Y_{rs}.   
\end{equation}
Using first identity  \eqref{def:dFl} and then Lemma \ref{lem:dFl}   we have
\[ 
\d (f \circ \Fl_{rt}^X)(Y) \circ \Fl^Y_{rs}=\d f (\d \Fl^X_{rt} (Y) ) \circ \Fl^X_{rt} \circ \Fl^Y_{rs}=\d f (Y) \circ \Fl^X_{rt} \circ \Fl^Y_{rs}
\]
and thus from \eqref{eq:derr} we obtain
\[
\frac\d{\d r} (f \circ G_r) = s\,\d f (Y) \circ \Fl^X_{rt} \circ \Fl^Y_{rs}+ t\,\d f(X) \circ \Fl^X_{rt} \circ \Fl^Y_{rs}=\d f(tX+sY)\circ G_r
\]
and since it is obvious by construction that $(G_r)$ has the properties $(i),(ii)$ in Definition \ref{def:RLF}, by Proposition \ref{prop:rlfeq} we deduce that $(G_r)$ is a Regular Lagrangian Flow of $tX+sY$ and thus by the uniqueness part of Theorem \ref{thm:AT} we deduce that for any $r\geq 0$ we have
\[
\Fl^{X}_{rt} \circ \Fl^Y_{rs} = \Fl_r^{tX+sY}
\]
$\mm$-a.e.. In particular this holds for $r=1$ and since both sides are continuous functions, equality holds everywhere.
\end{proof}
\section{Proof of the Main Result}

\subsection{Setting}\label{se:setting}
Here we fix the assumptions and notations that will be used in the rest of the text.

 $\mms$ is a $\RCD^\ast(0, N)$ metric measure space with $\supp(\mm)=\X$, $N \in \N$, $N>0$, and such that ${\rm dim}(H^k_{\rm dR})(\X) = N$.
 
The theory developed in \cite{Gigli14} grants the existence of $N$ harmonic vector fields $X_1,\ldots,X_N$ which are orthogonal in $\Lint^2(T\Sp)$. As in Section \ref{isox}, since the Ricci curvature is non-negative, these vector fields belong to  $H^{1,2}_C(T\Sp)$ and are parallel, i.e. $\nabla X_i\equiv 0$ for every $i$. It follows that  $\la X_i,X_j\ra\in \W(\Sp)$ with
\[
\d\langle X_i,X_j\rangle=\nabla X_i(\cdot,X_j)+\nabla X_j(\cdot,X_i)=0\qquad\mm-a.e.,
\]
which in turn  grants that $\la X_i,X_j\ra$ is $\mm$-a.e.\ equal to a constant function. Since $\int  \la X_i,X_j\ra\,\d\mm=0$ for $i\neq j$ we conclude that the $X_i$'s are pointwise orthogonal. The same argument also shows that up to normalization we  can, and will, assume that $|X_i|\equiv 1$ $\mm$-a.e.\ for every $i$. In particular, since these vector fields are in $L^2(T\X)$, we have
\begin{equation}
\label{eq:finmass}
\mm(\X)<\infty.
\end{equation}

\bigskip

We shall work with the product space $\Sp\times\R^N$ which will be equipped with the measure $\mm\times\mathcal L^N$ and the distance
\[
(\sfd\otimes\sfd_{\rm Eucl})^2\big((x,a),(y,b)\big):=\sfd^2(x,y)+|a-b|^2.
\]
We shall also define vector fields $Y_i\in L^0(\Sp\times\R^N)$, $i=1,\ldots,N$ as
\[
Y_i:=\Phi_2(\widehat{\nabla\pi_i})\qquad\forall i=1,\ldots,N,
\]
where $\pi_i:\R^N\to\R$ is the projection on the $i$-th coordinate,  $\widehat{\nabla\pi_i}\in L^0(\Sp,\Lint^0(T\R^N))$ is the function identically equal to $\nabla\pi^i\in \Lint^0(T\R^N)$ and where $\Phi_2:L^0(\Sp,\Lint^0(T\R^N))\to \Lint^0(T(\Sp\times\R^N))$ is defined in Proposition \ref{prop:defPhi}.

We also define the map $\F : \Sp \times \real^N \rightarrow \Sp$ by
\begin{equation}
\begin{split}
\F \colon \Sp \times \real^N &\rightarrow \Sp\\
(x,  \underline{a} =  (a_1, \dots a_N)) &\mapsto \Fl^{(X_1)}_{a_1} \circ \dots \circ \Fl^{(X_N)}_{a_N} (x).
\end{split}
\end{equation}
\subsection{Preliminary considerations}
Let us collect some easy consequences of our assumptions that can be derived from the discussion made in the previous sections. Start recalling from  \cite{AmbrosioGigliSavare12} (see also \cite{AmbrosioGigliSavare11-2}) that the product of two $\RCD(0,\infty)$ spaces is also $\RCD(0,\infty)$,  so that $\X\times\R^N$ is $\RCD(0,\infty)$.

From the fact that the $\nabla \pi_i$ are a pointwise ortonormal base for $L^0(T\R^N)$ and the fact that $\Phi_2$ preserves the pointwise norm we deduce that
\[
\la Y_i,Y_j\ra=\delta_{ij}\quad\mm\times\mathcal L^N-a.e.\quad \forall i,j
\]
and from the very definition of $\Phi_2$ we have that
\begin{equation}
\label{eq:yigrad}
Y_i=\nabla(\pi_i\circ\pi^{\R^N}).
\end{equation}
Since $\pi_i:\R^N\to\R$ is harmonic, we have $\nabla \pi_i\in D(\div_\loc,\R^N)$ with $\div(\nabla\pi_i)=0$ and thus from Propositions \ref{divvf} and \ref{prop:prodrcd} we deduce that $Y_i\in D(\div_\loc,\X\times\R^N)$ with 
\begin{equation}
\label{eq:div0}
\div Y_i\equiv 0.
\end{equation}
Taking into account \eqref{eq:lhodge} we also obtain that $Y_i^\flat\in D(\Delta_{H,\loc})$ with $\Delta_HY_i^\flat=0$ and since  $|\nabla \pi_i|\equiv 1$ and  $\Phi_2$ preserves the pointwise norm we also deduce that $|Y_i|\equiv 1$: these facts together with \eqref{eq:bocgen} grant that $Y_i\in W^{1,2}_{C,\loc}(T\X)$ with 
\begin{equation}
\label{eq:nabla0}
\nabla Y_i\equiv 0.
\end{equation}

\bigskip

Concerning the map $\F$, from Theorem \ref{thm:sumflow} we deduce that 
\begin{equation}
\label{eq:azione}
\F(\F(x,\underline a),\underline b)=\F(x,\underline a+\underline b)\qquad\forall x\in\Sp,\ \underline a,\underline b\in\R^N,
\end{equation}
so that we shall occasionally think at $\F$ as an action of $\R^N$ on $\X$. Theorem  \ref{thm:isofl} grants that this action is made of isometries, i.e.
\begin{equation}
\label{eq:isoaz}
\F(\cdot,\underline a)\colon\Sp\to\Sp\quad\text{ is an isometry for any }\underline a\in\R^N.
\end{equation}
From Theorem \ref{thm:sumflow} we also have
\[
\F(x,\underline a)= \Fl_1^{(X_{\underline a})} (x)\qquad\text{for }\qquad X_{\underline a}:=\sum_{i=1}^N a_i X_i
\]
and since the pointwise orthonormality of the $X_i$'s gives  $|X_{\underline a}|^2=|\sum_{i=1}^N a_i X_i|^2=\sum_i|a_i|^2=|\underline a|^2$ $\mm$-a.e., from \eqref{eq:metrsp} we deduce that for $\mm$-a.e.\ $x$ it holds
\[
\sfd(x,\F(x,\underline a))\leq\int_0^1|X_{\underline a}|\circ\Fl^{(X_{\underline a})}_t\,\d t\leq \||X_{\underline a}|\|_{L^\infty}=|\underline a|.
\]
Now the continuity of $\F(\cdot,\underline a)$ ensures that the above holds for every $x\in\X$ and thus taking \eqref{eq:azione} into account we conclude that 
\begin{equation}
\label{eq:1lip}
\F(x,\cdot)\colon\R^N\to\Sp\quad\text{ is 1-Lipschitz for any }x\in\Sp.
\end{equation}

Finally we remark that Proposition \ref{MP} together with Fubini theorem guarantees that
\begin{equation}\label{eq:Tpush}
\F_{\ast}(\mea \times \Leb^N\restr A) = \Leb^N(A) \mea,
\end{equation}
for every $A\subset \R^N$ Borel. This identity, \eqref{eq:isoaz} and \eqref{eq:1lip} grant in particular that $\F:\X\times\R^N\to\X$ is of local bounded deformation.

\subsection{An explicit formula for Regular Lagrangian Flows on $\Sp$}
Aim of  this section is to provide, in Proposition \ref{prop:reprfor}, an explicit representation formula for Regular Lagrangian Flows on $\Sp$ in terms of the map $\F$. The starting point is the following:
\begin{proposition}[Conjugation property]
With the same notations and assumptions as in Section \ref{se:setting} the following holds. For any $i=1,\ldots,N$, $t\in\R$ and $f\in\W_\loc(\Sp)$ it holds
\begin{equation}
\label{eq:conj1}
\d(f\circ\F)(Y_i)=\d f(X_i)\circ\F\qquad\mm\times\mathcal L^N-a.e..
\end{equation}
\end{proposition}
\begin{proof} Let us put  $\bar\Fl^i_t(x,\underline a):=(x,\underline a+te_i)$ and notice that by the very definition of $\F$ and identity\eqref{eq:azione}  we have
\begin{equation}
\label{eq:conj2}
\Fl^{X_i}_t\circ\F=\F\circ\bar\Fl^i_t.
\end{equation}
Fix  $f\in\W_\loc(\Sp)$ and recall Proposition \ref{prop:rlfeq}  to get
\[
\d f(X_i)\circ\F=\Big(\lim_{t\downarrow 0}\frac{f\circ\Fl^{(X_i)}_t-f}{t}\Big)\circ\F\stackrel{\eqref{eq:conj2}}=\lim_{t\downarrow 0}\frac{f\circ\F\circ \bar\Fl^i_t-f\circ\F}{t},
\]
the first limit being in $L^2_\loc(\Sp)$ and the second in $L^2_\loc(\Sp\times\R^N)$. Hence to conclude it is sufficient to show that for any $\tilde f\in\W_\loc(\Sp\times\R^N)$ and $\rho\in L^\infty(\Sp\times\R^N)$ with bounded support we have
\begin{equation}
\label{eq:conj3}
\lim_{t\downarrow0}\int\frac{\tilde f \circ \bar\Fl^i_t-\tilde f }{t}\rho\,\d(\mm\times\mathcal L^N)=\int  \d\tilde f(Y_i)\rho\,\d(\mm\times\mathcal L^N).
\end{equation}
By the linearity in $\rho$ of this expression we can further assume that $\rho$ is a probability density. Then put $\mu:=\rho\mm$ and  $\ppi:=(\bar\Fl^i_\cdot)_*\mu$, where here $\bar\Fl^i_\cdot:\Sp\times\R^N\to C([0,1],\Sp\times\R^N)$ is the map sending $(x,\underline a)$ to the curve $[0,1]\ni t\mapsto \Fl^i_t(x,\underline a)$. Notice that $\ppi$ is a test plan on $\Sp\times\R^N$ which  is concentrated on curves with speed constantly equal to 1, thus for any $\tilde f\in\W_\loc(\Sp\times\R^N)$ we have
\[
\begin{split}
\frac{\int \tilde f\,\d(\bar\Fl^i_t)_*\mu-\int \tilde f\,\d\mu }t&=\frac1t\int \tilde f(\gamma_t)-\tilde f(\gamma_0)\,\d\ppi(\gamma)\\
&\leq \frac1t\iint_0^t|\d \tilde f|(\gamma_s)|\dot\gamma_s|\,\d s\,\d\ppi(\gamma)\\
&\leq \frac1{2t}\iint_0^t|\d \tilde f|^2(\gamma_s)\,\d s\,\d\ppi(\gamma)
+\frac1{2t}\iint_0^t|\dot\gamma_s^2|\,\d s\,\d\ppi(\gamma)\\
&=\frac1{2t}\int_0^t\int|\d \tilde f|^2\circ \bar\Fl^i_s\,\d\mu\,\d s+\frac12.
\end{split}
\]
Recalling \eqref{eq:contl2loc} we thus have
\begin{equation}
\label{eq:lims}
\lims_{t\downarrow0}\frac{\int \tilde f\,\d(\bar\Fl^i_t)_*\mu-\int \tilde f\,\d\mu }t\leq \frac12\int |\d \tilde f|^2\,\d\mu+\frac12.
\end{equation}
Now put for brevity  $f_i:=\pi^i\circ\pi^{R^N}$, so that $f_i$ is 1-Lipschitz, \eqref{eq:yigrad} reads as
\begin{equation}
\label{eq:gradf}
\nabla f_i=Y_i
\end{equation} 
and  by construction it holds $f_i\circ\bar\Fl^i_s =f_i+s$, so that
\begin{equation}
\label{eq:limi}
\limi_{t\downarrow0}\frac{\int f_i\,\d(\bar\Fl^i_t)_*\mu-\int f_i\,\d\mu }t\geq 1\geq \frac12\int |\d f_i|^2\,\d\mu+\frac12.
\end{equation}
(In the terminology of \cite{Gigli12} we just proved that $\ppi$ represents the gradient of $f_i$ and we are now going to use the link between `horizontal and vertical' derivatives).
Writing \eqref{eq:lims} for $f+\eps \tilde f$ in place of $\tilde f$ and subtracting \eqref{eq:limi} we obtain
\[
\lims_{t\downarrow0}\eps\frac{\int \tilde f\,\d(\bar\Fl^i_t)_*\mu-\int \tilde f\,\d\mu }t\leq \frac12\int |\d (f+\eps \tilde f)|^2-|\d f|^2\,\d\mu\stackrel{\eqref{eq:gradf}}=\int \eps\,\d \tilde f(Y_i)+\frac{\eps^2}2|\d \tilde f|^2\,\d\mu.
\]
Dividing by $\eps>0$ (resp. $\eps<0$) and letting $\eps\downarrow 0$ (resp. $\eps\uparrow0$) we obtain \eqref{eq:conj3} and the conclusion.
\end{proof}

We now introduce a  map $\Psi:L^0(T\Sp)\to L^0(T(\Sp\times\R^N))$ as 
\begin{equation}
\label{eq:defphi}
\Psi(v):=\sum_{i=1}^N\la v,X_i\ra\circ\F\ \,Y_i
\end{equation}

\begin{lemma}\label{le:Psi1}
With the same notations and assumptions as in Section \ref{se:setting} and with $\Psi$ defined as in \eqref{eq:defphi}, the following holds. Let $v\in L^\infty\cap W^{1,2}_C(T\Sp)$. Then:
\begin{itemize}
\item[i)] $\la v,X_i\ra\in W^{1,2}(\Sp)$ for every $i$ and  $v\in D(\div)$ with
\begin{equation}
\label{eq:divv}
\div(v)=\sum_{i=1}^N\d (\la v,X_i\ra)(X_i)
\end{equation}
\item[ii)] $\Psi(v)\in L^\infty\cap W^{1,2}_{C,\loc}\cap D(\div_\loc)(T(\Sp\times\R^n))$ with
\[
\begin{split}
\nabla(\Psi(v))&=\sum_{i=1}^N\nabla(\la v,X_i\ra\circ\F )\otimes Y_i,\\
\div(\Psi(v))&=\div (v)\circ\F.
\end{split}
\]
\end{itemize}
\end{lemma}
\begin{proof}\\
\noindent{\bf (i)}
From \cite{Gigli14} we know that the assumptions on $v$ grant that $\la v,X\ra\in W^{1,2}_\loc(\Sp)$ for every $X\in L^\infty\cap H^{1,2}_C(T\Sp)$ with  
\[
\d\langle v,X\rangle=\nabla v(\cdot,X)+\nabla X(\cdot, v).
\]
Picking $X:=X_i$  and recalling that $D(\Delta_H)\subset (H^{1,2}_C(T\X))^\flat$ by the very definition of $\Delta_H$, we conclude that $\langle v,X_i\rangle$ belongs to $\W(\Sp)$, as claimed. Now put $a_i:=\la v,X_i\ra$ for brevity, so that $v=\sum_ia_iX_i$, let $f\in W^{1,2}(\Sp)$ and notice that
\[
\int \d f(v)\,\d\mm=\sum_{i=1}^N\int \d f(a_iX_i)\,\d\mm=-\sum_{i=1}^N\int f\div(a_iX_i)=-\int\sum_{i=1}^N f\d a_i(X_i)\,\d\mm,
\]
having used the fact that $\div(X_i)=0$. This proves both $v\in D(\div)$ and \eqref{eq:divv}.

\noindent{\bf (ii)}  The assumption $v\in L^\infty(T\Sp)$ trivially yields  $a_i\in L^\infty(\Sp)$ and since $Y_i\in L^\infty\cap W^{1,2}_C(\Sp\times\R^N)$ with $\nabla Y_i=0$ (recall \eqref{eq:nabla0}),  we have $(a_i\circ\F)\,Y_i\in W^{1,2}_{C,\loc}(T(\X\times\R^N))$ with
\[
\nabla((a_i\circ\F)\,Y_i)=\nabla(a_i\circ\F)\otimes Y_i+a_i\circ\F\,\nabla Y_i=\nabla(a_i\circ\F)\otimes Y_i.
\]
The fact that $\Phi(v)\in L^\infty\cap W^{1,2}_{C,\loc}(T(\Sp\times\R^n)$ and the formula for $\nabla(\Psi(v))$ follow.

We turn to the divergence: for   $g\in \W(\Sp\times\R^N)$ with bounded support we have
\[
\begin{split}
\int \d g(\Psi(v))\,\d(\mm\times\mathcal L^N)&=\sum_{i=1}^N\int \d g(a_i\circ\F\,Y_i)\,\d(\mm\times\mathcal L^N)\\
\text{because $\div (Y_i)=0$}\qquad\qquad&=-\sum_{i=1}^N\int g \,\d(a_i\circ\F)(Y_i)\,\d(\mm\times\mathcal L^N)\\
\text{by \eqref{eq:conj1}}\qquad\qquad&=-\sum_{i=1}^N\int g \,\d a_i(X_i)\circ\F\,\d(\mm\times\mathcal L^N),
\end{split}
\]
which by \eqref{eq:divv} is the conclusion.
\end{proof}
Let $(v_t)\in L^\infty([0,1],L^2(T\Sp))\cap L^2([0,1],W^{1,2}_C(T\Sp))$ be such that $(\div(v_t))\in L^\infty([0,1],L^\infty(\Sp))$, so that in particular the Regular Lagrangian Flow $(\Fl^{(v_t)}_s)$ is well defined. The  integrability condition of $(v_t)$ ensures that $(\la v_t,X_i\ra )\in L^\infty([0,1],L^2(\Sp))$ for every $i=1,\ldots,N$ and thus from \eqref{eq:bdflow} we see that $(\la v_s,X_i\ra \circ\Fl^{(v_t)}_s )\in L^\infty([0,1],L^2(\Sp))$ as well. Hence  the functions
\begin{equation}
\label{eq:Ait}
A_{i,t}:=\int_0^t\la v_s,X_i\ra \circ\Fl^{(v_t)}_s\,\d s\in L^2(\Sp),\qquad t\in[0,1],\ i=1,\ldots,N,
\end{equation}
are well defined. We then have the following result:

\begin{lemma}\label{le:psi2} With the same assumptions and notation as in  Section \ref{se:setting}, let $(v_t)\in L^\infty([0,1],L^2(T\Sp))\cap L^2([0,1],W^{1,2}_C(T\Sp))$ be such that $(\div(v_t))\in L^\infty([0,1],L^\infty(\Sp))$ and define $\Psi$ as in \eqref{eq:defphi}.

Then  the vector fields $\Psi(v_t)$ satisfy the assumptions of Theorem \ref{thm:AT} and for any $s\in\R$ the following identities hold $\mm\times\mathcal L^N$-a.e.:
\begin{align}
\label{eq:conj5}
\F\circ \Fl^{(\Psi(v_t))}_s&=\Fl^{(v_t)}_s\circ\F,\\
\label{eq:proj1}
\pi^{\Sp}\circ\Fl^{(\Psi(v_t))}_s  &=\pi^\Sp\\
\label{eq:proj2}
\pi_i\circ \pi^{\R^N}\circ \Fl^{(\Psi(v_t))}_s  &=\pi_i\circ \pi^{\R^N}+A_{i,s}\circ\F.
\end{align}
\end{lemma}
\begin{proof} The fact that the $\Psi(v_t)$'s satisfy the assumptions of Theorem  \ref{thm:AT}  is a direct consequence of the assumptions and Lemma \ref{le:Psi1}. Also, from \eqref{eq:conj1} we directly deduce
\begin{equation}
\label{eq:conj4}
\d (f\circ\F)(\Psi(v_t))=\d f(v_t)\circ\F\qquad\mm\times\mathcal L^N-a.e.
\end{equation}
for every $t\in[0,1]$ and $f\in\W_\loc(\Sp)$. Now pick $\bar\mu_0\in\prob{\Sp\times\R^N}$ with bounded support and such that $\bar\mu_0\leq C\mm\times\mathcal L^N$ for some $C>0$ and for $s\in\R$ define
\[
\bar\mu_s:=(\Fl^{(\Psi(v_t))}_s)_*\bar\mu_0\in\prob{\Sp\times\R^N}\qquad\qquad\text{ and }\qquad\qquad \mu_s:=\F_*\bar\mu_s\in\prob{\Sp}.
\]
We claim that $(\mu_s)$ solves the continuity equation with vector fields $(v_s)$ in the sense of Definition \ref{def:solce} and start observing that, locally in $s$, the measures $\bar\mu_s,\mu_s$ have uniformly bounded density. Now  pick $f\in\W(\Sp)$, so that   $f\circ\F\in\W_\loc(\Sp\times\R^N)$ and, from  Proposition \ref{prop:rlfeq}, $s\mapsto\int f\,\d\mu_s= \int f\circ\F\circ (\Fl^{(\Psi(v_t))}_s)\,\d\bar\mu_0$ is Lipschitz with 
\[
\begin{split}
\frac\d{\d s}\int f\,\d\mu_s&=\frac\d{\d s}\int f\circ\F\circ (\Fl^{(\Psi(v_t))}_s)\,\d\bar\mu_0\\
\text{by Proposition \ref{prop:rlfeq}}\qquad\qquad&=\int \d(f\circ\F)(\Psi(v_s))\circ (\Fl^{(\Psi(v_t))}_s)\,\d\bar\mu_0\\
&=\int \d(f\circ\F)(\Psi(v_s))\,\d\bar\mu_s\\
\text{ by  \eqref{eq:conj4}}\qquad\qquad&=\int \d f(v_s)\circ\F\,\d\bar\mu_s\\
&=\int \d f(v_s)\,\d\mu_s.
\end{split}
\]
This proves our claim. Hence  by the representation formula in Theorem \ref{thm:unice} we deduce that
\[
\F_*(\Fl^{(\Psi(v_t))}_s)_*\bar\mu_0=(\Fl^{(v_t)}_s)_*\F_*\bar\mu_0\qquad\forall s\in\R
\]
and from the arbitrariness of $\bar\mu_0$ and \eqref{eq:mappeuguale} identity \eqref{eq:conj5} follows.

To prove \eqref{eq:proj1} pick $\bar\mu_0,f$ and define $\bar\mu_s$ as above. Then we also put
\[
\nu_s:=\pi^\Sp_*\bar\mu_s\in\prob{\Sp}\qquad\forall s\in\R
\] 
and notice that again the $\nu_s$'s have, locally in $s$, uniformly bounded  densities and that it holds
\[
\begin{split}
\frac{\d}{\d s}\int f\,\d\nu_s=\frac{\d}{\d s}\int f\circ\pi^\Sp\,\d\bar\mu_s=\int \d(f\circ\pi^\Sp)(\Psi(v_s))\,\d\bar\mu_s=\int\Phi_1(\widehat{\d f}) ( \Psi(v_s))\,\d\bar\mu_s=0,
\end{split}
\]
where as usual $\widehat{\d f}\in L^0(\R^N,L^0(T^*\Sp))$ is the function identically equal to $\d f$ and the last identity follows from \eqref{eq:orto} and the very definitions of $\Psi(v_t)$ and $Y_i$. This shows that $(\nu_s)$ solves the continuity equation \eqref{eq:ce} with 0 vector fields, hence by the uniqueness of the solutions we deduce that $(\nu_s)$ is constant, i.e.
\[
\pi^\Sp_*(\Fl^{(\Psi(v_t))}_s)_*\bar\mu_0=\pi^\Sp_*\bar\mu_0,
\]
so that again the arbitrariness of $\bar\mu_0$ and  \eqref{eq:mappeuguale}  give \eqref{eq:proj1}.

For \eqref{eq:proj2}, we notice that the two sides agree for $s=0$ and are  absolutely continuous as functions of $s$ with values in $L^2_\loc(\Sp\times\R^N)$ (recall Proposition \ref{prop:rlfeq}). The conclusion then follows recalling  that  it holds $\nabla(\pi_i\circ \pi^{\R^N})=Y_i$,  so that
\[
\begin{split}
\frac{\d}{\d s} \pi_i\circ \pi^{\R^N}\circ \Fl^{(\Psi(v_t))}_s&=\d ( \pi_i\circ \pi^{\R^N})(\Psi(v_s))\circ \Fl^{(\Psi(v_t))}_s=\la Y_i,\Psi(v_s)\ra\circ \Fl^{(\Psi(v_t))}_s\\
&=\la v_s,X_i\ra\circ\F\circ \Fl^{(\Psi(v_t))}_s\stackrel{\eqref{eq:conj5}}=\la v_s,X_i\ra\circ \Fl^{( v_t)}_s\circ\F=\frac{\d}{\d s}A_{i,s}\circ\F.
\end{split}
\]
This is sufficient to conclude.
\end{proof}
We can now state the main result of the section:
\begin{proposition}[Representation formula for $\Fl^{(v_t)}$]\label{prop:reprfor}
With the same assumptions and notation as in  Section \ref{se:setting}, let $(v_t)\in L^\infty([0,1],L^2(T\Sp))\cap L^2([0,1],W^{1,2}_C(T\Sp))$ be such that $(\div(v_t))\in L^\infty([0,1],L^\infty(\Sp))$ and define the functions $A_{i,t}\in L^2(\Sp)$ as in \eqref{eq:Ait}. 

Then for any $s\in[0,1]$ and $\mm$-a.e.\ $x\in\X$ it holds
\begin{equation}
\label{eq:reprform}
\Fl^{(v_t)}_s(x)=\F(x,\underline A_s(x)),
\end{equation}
where $\underline A_s:=(A_{1,s},\ldots,A_{N,s})$.
\end{proposition}
\begin{proof} The identities \eqref{eq:proj1}, \eqref{eq:proj2} give
\[
\Fl^{(\Psi(v_t))}_s(x,\underline a)=\big(x,\underline a+\underline A_s(\F(x,\underline a))\big)\qquad\mm\times\mathcal L^N-a.e.\ (x,\underline a).
\]
Applying $\F$ on both sides and taking into account \eqref{eq:conj5} and \eqref{eq:azione} we obtain
\begin{equation}\label{rfa}
\Fl^{(v_t)}_s(\F(x,\underline a))=\F\big(\F(x,\underline a),\underline A_s(\F(x,\underline a))\big)\qquad\mm\times\mathcal L^N-a.e.\ (x,\underline a).
\end{equation}
Thus for any $A\subset \R^N$ Borel we have that \eqref{eq:reprform} holds for  $\F_*(\mm\times\mathcal L^N\restr{A})$-a.e.\ $x\in\X$
and the conclusion follows from \eqref{eq:Tpush}.
\end{proof}

\subsection{Further properties of $\F$ and conclusion}

We shall need the following result, proved in \cite{Gigli-Mosconi14}, about $W_2$-geodesics and continuity equation. Recall that a measure $\ppi\in\prob{C([0,1],\X)}$ is called \emph{lifting} of the geodesic $(\mu_t)$ provided
\[
\begin{split}
(\e_t)_*\ppi&=\mu_t\qquad\forall t\in[0,1],\\
\iint_0^1|\dot\gamma_t|^2\,\d t\,\d\ppi(\gamma)&<+\infty
\end{split}
\]
and that, on an $\RCD^*(K,N)$ space, as soon as either $\mu_0$ or $\mu_1$ is absolutely continuous w.r.t.\ $\mm$, there is a unique geodesic connecting them and a unique lifting of it (see \cite{GigliRajalaSturm13}).
\begin{proposition}\label{prop:LS} Let $(\X,\sfd,\mm)$ be a $\RCD^*(K,N)$ space and $(\mu_t)\subset \prob{\Sp}$ be a $W_2$-geodesic such that $\mu_0,\mu_1$ have both bounded support and  density.  Then there are vector fields $(v_t)\subset L^2(\Sp)$ such that:
\begin{itemize}
\item[i)] the continuity equation \eqref{eq:ce} is satisfied for $(\mu_t,v_t)$ in the sense of Definition \eqref{def:solce}
\item[ii)] letting $\ppi\in\prob{C([0,1],\Sp)}$ be the lifting of $(\mu_t)$ it holds
\begin{equation}
\label{eq:speedvt}
|v_t|(\gamma_t)=|\dot\gamma_t|\qquad\ppi\times\mathcal L^1-a.e.\ (\gamma,t),
\end{equation}
\item[iii)] $v_t\in H^{1,2}\cap D(\div)(\Sp)$ for every $t\in(0,1)$ and for every $\eps\in(0,1/2)$ it holds
\[
\int_\eps^{1-\eps}\||\nabla v_t|_\HS\|^2_{L^2(\Sp)}\,\d\mm+\sup_{t\in(\eps,1-\eps)}\Big(\|v_t\|_{L^\infty(\Sp)}+\|\div(v_t)\|_{L^\infty}\Big)<\infty
\]
\end{itemize}
\end{proposition}
We briefly comment this statement, given that in  \cite{Gigli-Mosconi14} it is not presented in this form. The vector fields $v_t$ are obtained as gradients of solutions $\eta_t$ of a double obstacle problem,  the obstacles being given by appropriate `forward' and `backward' Kantorovich potentials. This grants that $(i)$ holds. Then $(ii)$ is a general property of `optimal' lifting of solutions of the continuity equation (see e.g.\ \cite{Gigli14}). The estimates in $(iii)$ are the main gain from \cite{Gigli-Mosconi14}: the Laplacian comparison for the squared distance and the Lewy-Stampacchia inequality  grant the claimed uniform control on $\div(v_t)=\Delta\eta_t$. With a cut-off procedure based on the fact that $\mu_0,\mu_1$ are assumed to have bounded support, one can show that the $\eta_t$'s can be chosen to also have uniformly bounded support: this  and the $L^\infty$-bound on the Laplacian implies an $L^2$-bound on the Laplacian itself, so that from \eqref{eq:boundhess} we get the $L^2$-control on $|\nabla v_t|_\HS=|\H{\eta_t}|_\HS$. Finally, in  \cite{Gigli-Mosconi14} it has been proved that the $\eta_t$'s are Lipschitz and, although not explicitly mentioned, keeping track of the various constants involved one can see that it is provided  a uniform control on the Lipschitz constant for $t\in(\eps,1-\eps)$, which in turn implies the desired $L^\infty$ control on $|v_t|$.

Thanks to this result we can now prove the following crucial statement:
\begin{proposition}\label{prop:suriet}
With the same notations and assumptions as in Section \ref{se:setting} the following holds. For every $x,y\in\Sp$ there exists $\underline a\in\R^N$ such that
\[
\F(x,\underline a)=y\qquad\text{ and }\qquad |\underline a|\leq \sfd(x,y).
\]
\end{proposition}
\begin{proof}
Fix $y\in\Sp$, $R>0$, define
\[
\mu_0:=\mm(B_R(y))^{-1}\mm\restr{B_R(y)}
\qquad\qquad\qquad\mu_1:=\delta_y
\]
and let $(\mu_t)$ be the unique $W_2$-geodesic connecting $\mu_0$ to $\mu_1$ and $\ppi$ its lifting. Also, fix $\eps\in(0,1/2)$. Then  we know from \cite{GigliRajalaSturm13} that the $W_2$-geodesic $t\mapsto\mu^\eps_t:=\mu_{\eps+(1-2\eps)t}$ satisfies the assumptions of Proposition \ref{prop:LS} and that its lifting $\ppi^\eps$ is given by
 \[
\ppi^\eps= ({\rm Restr}_\eps^{1-\eps})_*\ppi, 
\]
 where ${\rm Restr}_{t_0}^{t_1}:C([0,1],\Sp)\to C([0,1],\Sp)$ is given by
\[
{\rm Restr}_{t_0}^{t_1}(\gamma)_t:=\gamma_{(1-t)t_0+tt_1}\qquad\forall \gamma\in C([0,1],\Sp).
\]
Up to pass to a further restriction,  Proposition \ref{prop:LS} grants the existence of vector fields $(v^\eps_t)$ satisfying $(i),(ii),(iii)$ in the statement. In particular, by $(iii)$ we know that the assumptions of Theorem \ref{thm:AT} are satisfied so that there exists the Regular Lagrangian Flow $(\Fl^{(v^\eps_t)}_s)$ of $(v_t^\eps)$. 

The representation formula for the solutions of the continuity equation given in  Theorem \ref{thm:unice} gives  
\begin{equation}
\label{eq:reprgeo}
\mu^\eps_s= (\Fl^{(v^\eps_t)}_s)_*\mu^\eps_0,\qquad\forall s\in[0,1].
\end{equation}
Thus letting $A^\eps_{i,t}$ be defined by \eqref{eq:Ait} for the vector fields $(v^\eps_t)$, from \eqref{eq:reprform} we deduce that 
\begin{equation}
\label{eq:flv}
\F(x,\underline A^\eps_1(x))=\Fl^{(v^\eps_t)}_1(x)\in\supp(\mu^\eps_1)\subset B_{\eps R}(y)\qquad \mu^\eps_0-a.e.\ x.
\end{equation}
Now notice that $\ppi$ is concentrated on constant speed geodesics of length bounded above by $R$, hence the same holds for $\ppi^\eps$, so that from \eqref{eq:speedvt} and \eqref{eq:reprgeo} we deduce that 
\begin{equation}
\label{eq:R}
|v_s^\eps|\circ\Fl^{(v^\eps_t)}_s\leq R\qquad \mu^\eps_0-a.e..
\end{equation}
Therefore using  the trivial inequality
\[
|\underline A_{1}^\eps|^2=\sum_{i=1}^N|\underline A_{i,1}^\eps|^2\leq \sum_{i=1}^N\int_0^1|\la v^\eps_s,X_i\ra|^2\circ \Fl^{(v^\eps_t)}_s\,\d s=\int_0^1|v^\eps_s|^2\circ \Fl^{(v^\eps_t)}_s\,\d s
\stackrel{\eqref{eq:R}}\leq R^2
\]
valid $\mu^\eps_0$-a.e.\ in conjunction with \eqref{eq:flv} we deduce that for $\mu^\eps_0=\mu_{\eps}$-a.e.\ $x$
\begin{equation}
\label{eq:exa}
\text{there exists $\underline a\in\R^N$ with $|\underline a|\leq R$ such that $\sfd(\F(x,\underline a),y)\leq \eps R$}
\end{equation}
and an  argument based on the continuity of $\F$ and the compactness of $B_R(0)\subset \R^N$ yields that the same holds for any $x\in\supp(\mu_\eps)$.

Now notice that simple considerations about the structure of $W_2$-geodesics grant that the Hausdorff distance between $\supp(\mu_0)=B_R(y)$ and $\supp(\mu_\eps)$ is bounded above by $\eps R$, thus for $x\in B_R(y)$ there is a sequence $n\mapsto x_n\in\supp(\mu_{1/n})$ converging to $x$. Let $\underline a_n$ be given by \eqref{eq:exa} for $x:=x_n$ and $\eps:=\frac1n$: by the uniform bound $|\underline a_n|\leq R$ and up to pass to a non-relabeled subsequence we can assume that $\underline a_n\to\underline a$ for some $\underline a\in \R^N$ with $|\underline a|\leq R$. Passing to the limit in
\[
\sfd\big(\F(x_n,\underline a_n),y\big)\leq \frac Rn
\]
using the continuity of $\F$ we conclude that $\F(x,\underline a)=y$. By the arbitrariness of $x\in B_R(y)$ and of $R>0$ the proof is completed.
\end{proof}
Let us now fix a point $\bar x\in\X$ and denote  by $\mathbb G\subset \R^N$ its stabilizer, i.e.
\begin{equation}
\label{eq:defg}
\mathbb G:=\Big\{\underline a\in\R^N\ :\ \F(\bar x,\underline a)=\bar x\Big\}.
\end{equation}
Notice that the last proposition (and the commutativity of $\R^N$) grants that the stabilizer does not depend on the choice of the particular point $\bar x$; moreover  $\mathbb G$ is a subgroup of $\R^N$ which, by the continuity of $\F$, is closed.

\begin{proposition}\label{prop:disc}
With the same notations and assumptions as in Section \ref{se:setting}  the following holds. The subgroup $\mathbb G$ of $\R^N$ defined in \eqref{eq:defg} is discrete.
\end{proposition}
\begin{proof}We argue by contradiction. If it is not discrete, being closed it must contain a line  so that for some $\underline a=(a_1,\ldots,a_N)\neq 0$ in $\R^N$ we have $t\underline a\in\mathbb G$ for every $t\in\R$. Put $X:=\sum_{i=1}^Na_iX_i$ and notice that $X$ is not identically 0 and harmonic, so that  its Regular Lagrangian Flow $(\Fl^{(X)}_t)$ consists of measure preserving isometries of $\Sp$ such that for $\mm$-a.e.\ $x$ the curve $t\mapsto \Fl^{(X)}_t(x)$ has constant positive speed. In particular, for $\mm$-a.e.\ $x$ such curve is not constant.

On the other hand, the very definition of $\F$ yields
\[
\F(x,t\underline a)=\Fl^{(X)}_t(x)\qquad\forall x\in\X,\ t\in\R
\]
and  by assumption the left hand side is equal to $x$ for every $t$: this gives the desired contradiction and the conclusion.
\end{proof}
The quotient space $\R^N/\mathbb G$ is  equipped with the only Riemannian metric letting the quotient map be a Riemannian submersion. The distance induced by this metric is
\begin{equation}
\label{eq:quotmetr}
\sfd_{\R^N/\mathbb G}\big([\underline a],[\underline b]\big)=\min_{\underline a':[\underline a']=[\underline a]\atop\underline b':[\underline b']=[\underline b]}|\underline a'-\underline b'|.
\end{equation}
Also, $\R^N/\mathbb G$ comes with a canonical, up to multiplication with a positive constant, reference measure $\mm_{\R^N/\mathbb G}$: the Haar measure, which also coincides with the volume measure induced by the metric.

Finally, the map $\F$ passes to the quotient and induces a map $\tilde \F:\R^N/\mathbb G\to\X$ via the formula:
\[
\tilde \F([\underline a]):=\F(\bar x,\underline a).
\]

With this said, we can now conclude the proof of our main result:
\begin{theorem}
With the same notations and assumptions as in Section \ref{se:setting}  the following holds. 
\begin{itemize}
\item[i)]The subgroup $\mathbb G$ of $\R^N$ defined in \eqref{eq:defg} is isomorphic to $\Z^N$, so that the quotient space $\R^N/\mathbb G$ is a flat torus $\mathbb T^N$.
\item[ii)] The induced quotient map $\tilde\F:\mathbb T^N\to \Sp$ is an isometry such that $\tilde \F_*\mm_{\mathbb T^N}=c\mm$ for some  $c>0$.
\end{itemize}
\end{theorem}
\begin{proof} \\
{\bf  $\tilde\F$ is an isometry}
From \eqref{eq:1lip} and the definition \eqref{eq:quotmetr} we get
\begin{equation}
\label{eq:facile}
\sfd\big(\tilde \F([\underline a]),\tilde \F([\underline b])\big)\leq\sfd_{\R^N/\mathbb G}\big( [\underline a], [\underline b]\big)\qquad\forall [\underline a], [\underline b]\in \R^N/\mathbb G.
\end{equation}
Now let $x,y\in\Sp$ and apply twice Proposition \ref{prop:suriet} to find $\underline a\in\R^N$ such that $\F(x,\underline a)=y$ and $|\underline a|\leq \sfd(x,y)$ and $\underline b\in\R^N$ such that $\F(\bar x,\underline b)=x$. Then we have
\[
\sfd_{\R^N/\mathbb G}\big( [\underline b], [\underline a+\underline b]\big)\leq |\underline a|\leq \sfd(x,y),
\]
and since by construction and from \eqref{eq:azione} we have $\tilde \F([\underline b])=x$ and $\tilde \F([\underline a+\underline b])=y$, this inequality   together with \eqref{eq:facile}  shows that $\tilde \F:\R^N/\mathbb G\to\Sp$ is an isometry.\\
{\bf  Up to a multiplicative constant, $\tilde\F$ is  measure preserving} Being an isometry, $\tilde\F$ is invertible: denote by   ${\sf S}:\Sp\to\R^N/\mathbb G$  its inverse and put $\mu:={\sf S}_*\mm$. For $\underline a=(a_1,\ldots,a_N)\in\R^N$ let $X_{\underline a}:=\sum_{i=1}^Na_iX_i$ and notice that \eqref{eq:azione} reads as $\F(\bar x,\underline b+\underline a)=\Fl_1^{(X_{\underline a})}(\F(\bar x,\underline b))$ for every $\underline a,\underline b\in\R^N$. Passing to the quotient we obtain
\begin{equation}
\label{eq:passquot}
\tilde\F([\underline b]+[\underline a])=\Fl_1^{(X_{\underline a})}(\tilde\F ([\underline b]))\qquad\forall \underline a,\underline b\in\R^N,
\end{equation}
hence letting $\tau^{[\underline a]}:\R^N/\mathbb G\to\R^N/\mathbb G$ be the translation by $[\underline a]$ defined by $\tau^{[\underline a]}([\underline b]):=[\underline b]+[\underline a]$ we can rewrite \eqref{eq:passquot} as
\begin{equation}
\label{eq:S}
\tau^{[\underline a]}\circ{\sf S}={\sf S}\circ\Fl^{(X_{\underline a})}_1\qquad\forall \underline a\in\R^N.
\end{equation}
Therefore we have
\[
\tau^{[\underline a]}_*\mu=\tau^{[\underline a]}_*{\sf S}_*\mm\stackrel{\eqref{eq:S}}={\sf S}_*(\Fl^{(X_{\underline a})}_1)_*\mm\stackrel{\eqref{eq:mespre}}={\sf S}_*\mm=\mu\qquad\forall \underline a\in\R^N.
\]
This shows that $\mu$ is translation invariant and thus a multiple of the Haar measure $\mm_{\R^N/\mathbb G}$. \\
{\bf  Up to a multiplicative constant, $\tilde\F$ is  measure preserving} What we just proved and   \eqref{eq:finmass} ensure that $\mm_{\R^N/\mathbb G}$ is a finite measure. Now recall that, as it is well known and trivial to prove, discrete subgroups of $\R^N$ are isomorphic to $\Z^n$ for some $n\leq N$ and that $\R^N/\Z^n$ has finite volume if and only if $n=N$. Being $\mathbb G$ discrete (Proposition \ref{prop:disc}), the thesis follows.
\end{proof}

\appendix
\section{Notes on the Hessian on product spaces}
In this appendix we continue the investigation of differential operators in product spaces by considering products of  $\RCD$ spaces and the Hessian of those functions  depending only on one variable. Recall from \cite{AmbrosioGigliSavare12} (see also \cite{AmbrosioGigliSavare11-2}) that the product of two $\RCD(K,\infty)$ spaces is $\RCD(K,\infty)$ and that the tensorization of the Cheeger energy in the sense of Definition \ref{def:tensch} holds.

On the other hand, it is not clear whether the density of the product algebra in the sense of \ref{def:densalg} holds or not, and in any case it seems that the following slightly stronger density property is necessary for the current purposes:
\begin{definition}[Density of the product algebra - strong form]\label{def:appprod}
We say that  two metric measure spaces $\mmso$ and $\mmst$ have the property of density of product algebra in the strong form if for $\mathcal A\subset W^{1,2}(\X_1\times\X_2)$ defined as in \eqref{eq:prodalg} it holds: for $f\in L^\infty\cap\W(\Sp_1 \times \Sp_2)$ there exists $(f_n)\subset \mathcal A$ uniformly bounded and $\W$-converging to $f$.
\end{definition}
\begin{remark}{\rm
If $X_1$ is infinitesimally Hilbertian and $X_2$ the Euclidean space, such strong form of density holds. This is a consequence of the construction done in \cite{GH15}, which grants that for $X_1$ arbitrary and $X_2=\R$, for any $f\in L^\infty\cap W^{1,2}(X_1\times \X_2)$ we can find $(f_n)\subset\mathcal A$ uniformly bounded and such that $(f_n),(|\d f_n|)$ converge to $f,|\d f|$ in $L^2$ respectively. The infinitesimal Hilbertianity of $\X_1$ and the tensorization of the Cheeger energy (proved in \cite{GH15}) implies the infinitesimal Hilbertianity of $\X_1\times\X_2$ and in turn this forces the $W^{1,2}$-convergence of the functions $(f_n)$ above to $f$.

The case $\X_2=\R^n$ then comes from an induction argument.
}\fr\end{remark}
This extra density assumption is needed in the following approximation lemma in order to use the $L^\infty-$Lip regularization of the heat flow (see \cite{AmbrosioGigliSavare11-2}). Such lemma is about approximation of test functions in the product with test functions depending on one variable only and in order to formulate the result it is convenient to introduce the algebra $\tilde {\mathcal A}$ as
\[
\tilde{\mathcal{A}} := \Big\{  \displaystyle \sum_{j=1}^n g_{1,j}\circ\pi_1\, g_{2,j}\circ\pi_2\ :\ n\in\N,\  \begin{array}{l}
g_{1,j} \in  \test{\Sp_1}\text{ has bounded support}\\
 g_{2,j} \in \test{\Sp_2} \text{ has bounded support}.
 \end{array}\Big\}
\]
Notice that the calculus rules obtained in Section \ref{se:calcprod}  ensure that $\mathcal{\tilde A}\subset \test{\Sp_1\times\Sp_2}$. 

We then have the following lemma about approximation of test functions with ones in $\tilde{\mathcal{A}}$; notice that a two-steps procedure is needed because the required uniform bound on the differentials prevents arguments by diagonalization.
\begin{lemma}\label{le:apprtest}
Let $\mmso$ and $\mmst$ be two $\RCD(K, \infty)$ metric measure spaces for which the density of the product algebra holds in the strong form (Definition \ref{def:appprod}). Let $f \in \test{\Sp_1\times\Sp_2}$ be with bounded support and find $\nchi_1\in \test{\Sp_1}$, $\nchi_2\in \test{\Sp_2}$  with bounded support and such that $\supp(f)$ is contained in the interior of $\{\nchi_1=1\}\times\{\nchi_2=1\}$ (recall \eqref{eq:testsupp}) and for $t>0$ put $\tilde f_t:=\nchi_1\circ\pi_1\nchi_2\circ\pi_2\,\h_tf$.

Then:
\begin{itemize}
\item[i)] It holds
\begin{itemize}
\item[a)] $\tilde f_t\to f$ in $W^{2,2}(\Sp_1\times\Sp_2)$ as $t\downarrow0$,
\item[b)] $\Delta\tilde f_t\to\Delta f$ in $L^2(\Sp_1\times\Sp_2)$ as $t\downarrow0$,
\item[c)] $\sup_{t\in(0,1)}\||\d\tilde f_t|\|_{L^\infty}<\infty$,
\item[d)] the sets $\supp(\tilde f_t)$ are uniformly bounded for $t\in(0,1)$,
\end{itemize}
\item[ii)] For every $t>0$ there exists a sequence $(g_n)\subset \tilde{\mathcal{A}}$ such that:
\begin{itemize}
\item[a)] $g_n\to\tilde f_t$ in $W^{2,2}(\Sp_1\times\Sp_2)$ as $n\to\infty$,
\item[b)] $\Delta g_n\to\Delta\tilde f_t$ in $L^2(\Sp_1\times\Sp_2)$ as $n\to\infty$,
\item[c)] $\sup_{n\in\N}\||\d\tilde g_n|\|_{L^\infty}<\infty$,
\item[d)] the sets $\supp(g_n)$ are uniformly bounded,
\end{itemize}
\end{itemize}
\end{lemma}
\begin{proof}\\
{\bf (i)} It is well known  that $\h_tf\to f$ in $\W(\Sp_1\times\Sp_2)$ and $\Delta\h_tf\to\Delta f$ in $L^2(\Sp_1\times\Sp_2)$ as $t\downarrow0$. From the Leibniz rules for the gradient and the Laplacian and taking into account Proposition \ref{prop:defPhi} and Corollary \ref{cor:lapprod} we then see that $\tilde f_t\to f$ in $W^{1,2}(\Sp_1\times\Sp_2)$ and $\Delta\tilde f_t\to\Delta f$ in $L^2(\Sp_1\times\Sp_2)$ as $t\downarrow0$. Convergence in $W^{2,2}$ then follows by \eqref{eq:boundhess}. The uniform bounds on the supports is trivial by construction and the uniform bound on the differential follows by the Bakry-\'Emery estimate (see Theorem 7.2 in \cite{AmbrosioGigliMondinoRajala12}) and the fact that $|\d f|\in L^\infty$.

\noindent{\bf (ii)} Fix $t>0$ and use \eqref{eq:testsupp} to find functions $\tilde\nchi_1\in \Test  (\Sp_1)$, $\tilde\nchi_2\in\Test  (\Sp_2)$  with bounded support such that $\supp(\tilde f_t)$ is contained in the interior of $\{\tilde\nchi_1=1\}\times\{\tilde\nchi_2=1\}$. Also, let $(f_n)\subset \mathcal{ A}$ be uniformly bounded and $W^{1,2}$-converging to $f$  and put 
\[
g_n:=(\nchi_1\tilde\nchi_1)\circ\pi_1\,(\nchi_2\tilde\nchi_2)\circ\pi_2\,\h_tf_n\qquad\forall n\in\N.
\]
We claim that the $g_n$'s satisfy the thesis. Indeed, from the regularizing properties of the heat flow we know that $\h_tf_n\to\h_tf$ in $W^{1,2}(\Sp_1\times\Sp_2)$ and $\Delta\h_tf_n\to\Delta\h_tf$ in $L^2(\Sp_1\times\Sp_2)$. Since $(\nchi_1\tilde\nchi_1)\circ\pi_1\,(\nchi_2\tilde\nchi_2)\circ\pi_2\,\h_tf=\tilde f_t$, the same arguments used in the previous step grant that $g_n\to \tilde f_t$ in $W^{2,2}(\Sp_1\times\Sp_2)$ and $\Delta g_n\to \Delta\tilde f_t$ in $L^2(\Sp_1\times\Sp_2)$. The fact that the supports of the $g_n$'s are uniformly bounded is obvious, and the uniform bound on the differentials follows from the uniform bounds on the $f_n$'s and the $L^\infty-\Lip$ regularization property  (see Theorem 7.3 in \cite{AmbrosioGigliMondinoRajala12}).

Thus it remains to show that $g_n\in\mathcal{\tilde A}$ and since test functions form an algebra, to this aim it is sufficient to show that $
\h_tf_n\in\mathcal{\tilde A}$. By the linearity of the heat flow, the fact that $\h_t h$ is a test function for $h\in L^\infty$ and $t>0$ and performing  if necessary a truncation argument on the various addends in $f_n\in\mathcal A$, to conclude it is sufficient to show that for $h_1\in L^\infty({\Sp_1})$ and $h_2\in L^\infty( \Sp_2)$  it holds
\begin{equation}
\label{eq:tenscal}
\h_t(h_1\circ\pi_1\,h_2\circ\pi_2)=h_1\circ\pi_1\, \h^{\Sp_2}_t(h_2)\circ\pi_2+h_2\circ\pi_2\, \h^{\Sp_1}_t(h_1)\circ\pi_2\qquad\forall t>0,
\end{equation}
where $\h_t^{\Sp_1}$, $\h_t^{\Sp_2}$ are the heat flows in $\Sp_1,\Sp_2$ respectively. To this aim notice that Corollary \ref{cor:lapprod} grants that for $h_1,h_2$ in the domain of the Laplacian in the respective spaces it holds
\[
\Delta (h_1\circ\pi_1\,h_2\circ\pi_2)=h_1\circ\pi_1\,(\Delta h_2)\circ\pi_2+h_2\circ\pi_2\,(\Delta h_1)\circ\pi_1
\]
then observe that thanks to this fact the map sending $t\geq 0$ to the right hand side of \eqref{eq:tenscal}, call it $\tilde h_t$, is absolutely continuous with values in $L^2(\X_1\times\X_2)$ and its derivative is given by $\Delta  \tilde h_t$. By the uniqueness of the heat flow we conclude that $\tilde h_t=\h_t(\tilde h_0)$, which is our claim.
\end{proof}
We then have the following result:
\begin{proposition}\label{hessprodprop}
Let $\mmso$ and $\mmst$ be two $\RCD(K, \infty)$  spaces for which the density of the product algebra holds in the strong form (Definition \ref{def:appprod})  and let $f \in \WT_\loc(\Sp_1)$. Then  $f\circ\pi_1 \in \WT_{\loc}(\Sp_1 \times \Sp_2)$ and
\begin{equation}\label{hessprod}
\hess(f\circ \pi_1)(\nabla g, \nabla\tilde g) (x_1,x_2)= \hess(f)(\nabla g_{x_2}, \nabla\tilde g_{x_2}) (x_1),\qquad\mm_1\times\mm_2-a.e.\ (x_1,x_2)
\end{equation}
for every $g,\tilde g \in \test{\X_1 \times \X_2}$. 
\end{proposition}
\begin{proof} It is readily verified that the map sending $g,\tilde g\in \Test(\X_1 \times \X_2)$ to the right hand side of \eqref{hessprod} defines an element of $L^2_\loc((T^*)^{\otimes 2}(\Sp_1\times\Sp_2))$, hence to conclude it is sufficient to show that for such element the identity \eqref{eq:defhess} holds.

Now consider the identity \eqref{eq:defhess} defining the Hessian for functions in $W^{2,2}_\loc$ with $g:=g_n$, where $(g_n)$  is a sequence of test functions $W^{2,2}$-converging to some limit $g$, such that $\Delta g_n\to\Delta g$ in $L^2$ and with $\supp(g_n)$ and $\||\d g_n|\|_{L^\infty}$ uniformly bounded: it is readily verified that in this case  the two sides of \eqref{eq:defhess}  pass to the limit.

Thus by  Lemma \ref{le:apprtest} above and the bilinearity and symmetry in $g,\tilde g$, to conclude it is sufficient to consider $g=\tilde g$ of the form $g=g_1\circ\pi_1\,g_2\circ\pi_2$ for $g_1 \in \test{\X_1}$ and $g_2 \in \test{\X_2}$ both with bounded support. For such $g$ we have $g_{x_2}=g_2(x_2)g_1$, and thus $\nabla g_{x_2}=g(x_2)\nabla g_1$, so that our aim is  to show that  for any $h\in \test{\X_1 \times \X_2}$ with bounded support it holds
\begin{equation}
\label{eq:perhess}
\begin{split}
-\int &\la\nabla(f\circ\pi_1),\nabla g\ra\div(h\nabla g)+h\langle\nabla (f\circ\pi),\nabla\frac{|\nabla g|^2}2 \rangle\,\d(\mm_1\times\mm_2)\\
&\qquad\qquad\qquad\qquad=\int h  g_2^2\circ\pi_2 \hess(f)(\nabla g_1, \nabla g_1) \circ\pi_1 \,\d(\mm_1\times\mm_2).
\end{split}
\end{equation}
Denoting for clarity  $\div_1,\div_2$ the divergence operators in $\Sp_1,\Sp_2$ respectively and using formulas \eqref{eq:pera}, \eqref{eq:orto} and \eqref{eq:divprod}, for $\mm_1\times\mm_2$-a.e.\ $(x_1,x_2)$ we have
\begin{equation}
\label{eq:h1}
\begin{split}
&\big(\la\nabla(f\circ\pi_1),\nabla g\ra\div(h\nabla g)\big)(x_1,x_2)\\
&\qquad=g_2(x_2)\la \nabla f,\nabla g_1\ra(x_1)\big(  g_2(x_2)\div_1(h_{x_2}\nabla g_1)(x_1)+ g_1(x_1)\div_2(h^{x_1}\nabla g_2)(x_2)\big).
\end{split}
\end{equation}
From \eqref{Chpr} we have
\[
|\nabla g|^2=g_2^2\circ\pi_2|\nabla g_1|^2\circ\pi_1+g_1^2\circ\pi_1|\nabla g_2|^2\circ\pi_2
\]
and thus recalling \eqref{eq:orto} we obtain
\[
\begin{split}
h\langle\nabla (f\circ\pi_1),\nabla\frac{|\nabla g|^2}2\rangle= h\,g^2_2\circ\pi_2\langle\nabla f,\nabla\frac{|\nabla g_1|^2}2\rangle\circ\pi_1+h\,|\nabla g_2|^2\circ\pi_2\big(g_1 \la\nabla f,\nabla g_1\ra\big)\circ\pi_1.
\end{split}
\]
Adding up this identity and \eqref{eq:h1} and integrating, the conclusion \eqref{eq:perhess} follows by the defining property \eqref{eq:defhess} of $\H f$ and the trivial identity
\[
\begin{split}
\int g_1(x_1)g_2(x_2)\la \nabla f,\nabla g_1\ra(x_1)&\div_2(h^{x_1}\nabla g_2)(x_2)\,\d\mm_1(x_1)\,\d\mm_2(x_2)\\
&=\int g_1(x_1)\la \nabla f,\nabla g_1\ra(x_1)\int g_2\div_2(h^{x_1}\nabla g_2)\,\d\mm_2\,\d\mm(x_1)\\
&=-\int g_1(x_1)\la \nabla f,\nabla g_1\ra(x_1)\int h^{x_1} |\nabla g_2|^2\,\d\mm_2\,\d\mm(x_1)\\
&=-\int h\,|\nabla g_2|^2\circ\pi_2\big(g_1 \la\nabla f,\nabla g_1\ra\big)\circ\pi_1\,\d(\mm_1\times\mm_2).
\end{split}
\]
\end{proof}

\def\cprime{$'$} \def\cprime{$'$}


\begin{thebibliography}{10}

\bibitem{Ambrosio04}
{\sc L.~Ambrosio}, {\em Transport equation and {C}auchy problem for {$BV$}
  vector fields}, Invent. Math., 158 (2004), pp.~227--260.

\bibitem{AmbrosioGigliMondinoRajala12}
{\sc L.~Ambrosio, N.~Gigli, A.~Mondino, and T.~Rajala}, {\em Riemannian {R}icci
  curvature lower bounds in metric measure spaces with $\sigma$-finite
  measure}, Trans. Amer. Math. Soc., 367 (2012), pp.~4661--4701.

\bibitem{AmbrosioGigliSavare11-3}
{\sc L.~Ambrosio, N.~Gigli, and G.~Savar{\'e}}, {\em Density of {L}ipschitz
  functions and equivalence of weak gradients in metric measure spaces}, Rev.
  Mat. Iberoam., 29 (2013), pp.~969--996.

\bibitem{AmbrosioGigliSavare11}
\leavevmode\vrule height 2pt depth -1.6pt width 23pt, {\em Calculus and heat
  flow in metric measure spaces and applications to spaces with {R}icci bounds
  from below}, Invent. Math., 195 (2014), pp.~289--391.

\bibitem{AmbrosioGigliSavare11-2}
\leavevmode\vrule height 2pt depth -1.6pt width 23pt, {\em Metric measure
  spaces with {R}iemannian {R}icci curvature bounded from below}, Duke Math.
  J., 163 (2014), pp.~1405--1490.

\bibitem{AmbrosioGigliSavare12}
\leavevmode\vrule height 2pt depth -1.6pt width 23pt, {\em Bakry-\'{E}mery
  curvature-dimension condition and {R}iemannian {R}icci curvature bounds}, The
  Annals of Probability, 43 (2015), pp.~339--404.

\bibitem{AmbrosioMondinoSavare13-2}
{\sc L.~Ambrosio, A.~Mondino, and G.~Savar{\'e}}, {\em On the {B}akry-\'{E}mery
  condition, the gradient estimates and the {L}ocal-to-{G}lobal property of
  ${RCD}^*({K}, {N})$ metric measure spaces}, The Journal of Geometric
 Analysis, 26 (2014), pp.~1--33.

\bibitem{Ambrosio-Trevisan14}
{\sc L.~Ambrosio and D.~Trevisan}, {\em Well posedness of {L}agrangian flows
  and continuity equations in metric measure spaces}, Anal. PDE, 7 (2014),
 pp.~1179--1234.

\bibitem{AT15}
\leavevmode\vrule height 2pt depth -1.6pt width 23pt, {\em Lecture notes on the
  {D}i{P}erna-{L}ions theory in abstract measure spaces}.
\newblock Accepted at Annales Fac. Sc. de Toulouse, arXiv:1505.05292, 2015.

\bibitem{Cheeger00}
{\sc J.~Cheeger}, {\em Differentiability of {L}ipschitz functions on metric
 measure spaces}, Geom. Funct. Anal., 9 (1999), pp.~428--517.

\bibitem{Colding97}
{\sc T.~H. Colding}, {\em Ricci curvature and volume convergence}, Ann. of
  Math. (2), 145 (1997), pp.~477--501.

\bibitem{DiPerna-Lions89}
{\sc R.~J. DiPerna and P.-L. Lions}, {\em Ordinary differential equations,
 transport theory and {S}obolev spaces}, Invent. Math., 98 (1989),
  pp.~511--547.

\bibitem{Gigli17}
{\sc N.~Gigli}, {\em Lecture notes on differential calculus on {${\rm RCD}$}
  spaces}.
\newblock Preprint, arXiv: 1703.06829.

\bibitem{Gigli13}
\leavevmode\vrule height 2pt depth -1.6pt width 23pt, {\em The splitting
  theorem in non-smooth context}.
\newblock Preprint, arXiv:1302.5555, 2013.

\bibitem{Gigli14}
\leavevmode\vrule height 2pt depth -1.6pt width 23pt, {\em Nonsmooth
  differential geometry - an approach tailored for spaces with {R}icci
  curvature bounded from below}.
\newblock Accepted at Mem. Amer. Math. Soc., arXiv:1407.0809, 2014.

\bibitem{Gigli13over}
\leavevmode\vrule height 2pt depth -1.6pt width 23pt, {\em An overview of the
 proof of the splitting theorem in spaces with non-negative {R}icci
  curvature}, Analysis and Geometry in Metric Spaces, 2 (2014), pp.~169--213.

\bibitem{Gigli12}
\leavevmode\vrule height 2pt depth -1.6pt width 23pt, {\em On the differential
  structure of metric measure spaces and applications}, Mem. Amer. Math. Soc.,
  236 (2015), pp.~vi+91.

\bibitem{GH15}
{\sc N.~Gigli and B.~Han}, {\em Sobolev spaces on warped products}.
\newblock Preprint, arXiv:1512.03177, 2015.

\bibitem{GigliHan13}
\leavevmode\vrule height 2pt depth -1.6pt width 23pt, {\em The continuity
  equation on metric measure spaces}, Calc. Var. Partial Differential
  Equations, 53 (2013), pp.~149--177.

\bibitem{Gigli-Mosconi14}
{\sc N.~Gigli and S.~Mosconi}, {\em The abstract {L}ewy-{S}tampacchia
  inequality and applications}, J. Math. Pures Appl. (9), 104 (2014),
  pp.~258--275.

\bibitem{GP16}
{\sc N.~Gigli and E.~Pasqualetto}, {\em Equivalence of two different notions of
 tangent bundle on rectifiable metric measure spaces}.
\newblock Preprint, arXiv:1611.09645, 2016.

\bibitem{GigliRajalaSturm13}
{\sc N.~Gigli, T.~Rajala, and K.-T. Sturm}, {\em Optimal {M}aps and
  {E}xponentiation on {F}inite-{D}imensional {S}paces with {R}icci {C}urvature
  {B}ounded from {B}elow}, J. Geom. Anal., 26 (2016), pp.~2914--2929.

\bibitem{Han14}
{\sc B.~Han}, {\em Ricci tensor on {${\rm RCD}^*(K,N)$} {spaces}}.
\newblock Preprint, arXiv: 1412.0441.

\bibitem{Honda14}
{\sc S.~Honda}, {\em Elliptic {PDE}s on compact {R}icci limit spaces and
  applications}.
\newblock Preprint, arXiv:1410.3296.

\bibitem{MondinoWei16}
{\sc A.~Mondino and G.~Wei}, {\em On the universal cover and the fundamental
  group of an {${\rm RCD}(K,N)$}-space}.
\newblock Accepted at Crelle's journal, arXiv: 1605.02854.

\bibitem{Petersen16}
{\sc P.~Petersen}, {\em Riemannian geometry}, vol.~171 of Graduate Texts in
 Mathematics, Springer, Cham, third~ed., 2016.

\bibitem{Savare13}
{\sc G.~Savar{\'e}}, {\em Self-improvement of the {B}akry-\'{E}mery condition
  and {W}asserstein contraction of the heat flow in {${\rm RCD}(K,\infty)$}
  metric measure spaces}, Discrete Contin. Dyn. Syst., 34 (2014),
  pp.~1641--1661.

\bibitem{Shanmugalingam00}
{\sc N.~Shanmugalingam}, {\em Newtonian spaces: an extension of {S}obolev
  spaces to metric measure spaces}, Rev. Mat. Iberoamericana, 16 (2000),
  pp.~243--279.

\end{thebibliography}
\end{document}